\documentclass[11pt,fleqn]{amsart} 

\usepackage[usenames,dvipsnames,condensed]{xcolor}

\usepackage{amsthm}
\usepackage{amsfonts}
\usepackage[english]{babel}
\usepackage[usenames]{xcolor}
\usepackage{graphicx}
\usepackage{soul}
\usepackage{stfloats}
\usepackage{morefloats}
\usepackage{cite}
\usepackage{lscape}
\usepackage{epstopdf}
\usepackage{braket}
\usepackage[lite]{amsrefs}
\usepackage{mathbbol}
\usepackage{tikz}
\usepackage{tikz-cd}
\usepackage{mathtools}
 \usepackage{overpic}

\usepackage{amsmath,amstext,amsopn,amsfonts,eucal,amssymb}
\usepackage{graphicx,wrapfig,url}

\newtheorem{theorem}{Theorem}[section]
\newtheorem{corollary}[theorem]{Corollary}
\newtheorem{lemma}[theorem]{Lemma}
\newtheorem{proposition}[theorem]{Proposition}
\newtheorem{definition}[theorem]{Definition}

\newtheorem{example}[theorem]{Example}

\newtheorem{remark}[theorem]{Remark}

\begin{document}

\title{Yang-Baxter Hochschild Cohomology} 

\author{Masahico Saito} 
\address{Department of Mathematics, 
	University of South Florida, Tampa, FL 33620, U.S.A.} 
\email{saito@usf.edu} 

\author{Emanuele Zappala} 
\address{Department of Mathematics and Statistics, Idaho State University\\
	Physical Science Complex |  921 S. 8th Ave., Stop 8085 | Pocatello, ID 83209} 
\email{emanuelezappala@isu.edu}

\maketitle

\begin{abstract}

Braided algebras are associative algebras endowed with a Yang-Baxter operator that satisfies certain compatibility conditions involving the multiplication. 
Along with Hochschild cohomology of  algebras, there is also a notion of Yang-Baxter cohomology, which is associated to any Yang-Baxter operator. 
In this article, we introduce and study a cohomology theory for braided algebras in dimensions 2 and 3, that unifies Hochschild and Yang-Baxter cohomology theories, and generalizes to all dimensions in characteristic $2$. We show that its
second cohomology group classifies infinitesimal deformations of braided algebras.
We provide infinite families of examples of braided algebras, including Hopf algebras, tensorized multiple conjugation quandles, and braided Frobenius algebras. Moreover, we derive the obstructions  to higher deformations, which lie in the third cohomology group. Relations to Hopf algebra cohomology are also discussed.
\end{abstract}

\date{\empty}

\tableofcontents

\section{Introduction}

Hochschild homology is defined for associative algebras, and has been studied extensively in deformation theories of algebras \cite{GS}. It corresponds to group homology in discrete case \cite{Brown}. 
Its 2-dimensional cohomology group, in particular,  describes extensions of groups.
A counterpart of group homology for self-distributive (SD) structures called racks and quandles
has been developed \cite{FR,CJKLS}. The SD structure is related to braiding, 
and  its homology theory has been  applied to knot theory \cite{FRS},
such as in quandle cohomology theories \cite{CJKLS}. Its 2-dimensional cohomology group describes extensions of racks and quandles as well.
Tensor versions of SD structure that correspond to Hochschild homology, and their homology theory, have also been developed
\cite{CCES-coalgebra,CCES-adjoint}.

The SD structure has been used to construct solutions to the Yang-Baxter equation (YBE)  \cite{AZ,CCES-coalgebra}.
The YBE has been studied extensively in physics and knot theory as well.
In particular in knot theory, solutions to YBE have been employed in the construction of quantum knot invariants. 
Counterparts of Hochschild homology for the Yang-Baxter equation have been also developed
in several ways.
A homology theory for set-theoretic YBE, the discrete case, was defined in \cite{CCES1}. 
Diagrammatic representations have been developed for this theory in \cite{Lebed,PW},
through which it was extended to the tensor case.
Interpretations of this tensor version 
do not seem clearly known at this time. 
A different homology theory for YBE was defined in \cite{Eisermann1,Eisermann}  from point of view of
the deformation theory. Relations between these two homology theories do not seem to be known.  However, a relation between $n$-Lie algebra cohomology, SD cohomology and the Yang-Baxter cohomology of Eisermann has recently been studied in \cite{El-Zap}.

Diagrammatics of associative and SD structures have been used effectively in both algebraic and topological context. When a binary  operation is represented by a trivalent vertex of a graph, associativity corresponds to the change of tree diagrams called {\it IH-move}. 
As diagrammatic representation of SD operations, crossings of knot diagrams have been employed,
that corresponds to the braiding and the YBE. 

On the other hand,  
handlebody-links and surface ribbons (compact orientable surfaces with boundary expressed in knotted ribbon forms) have been represented by 
spatial trivalent graphs, with sets of moves specified in each context. See \cite{Ishii08,CIST,Matsu}. 
Thus the idea of using operations that possess both associative and SD (partial) operations
has been explored \cite{CIST} in discrete case, and its homology theory has been developed.
A tensor version of such structure was proposed in \cite{SZbrfrob}, called {\it braided Frobenius algebras}, and a construction of such algebras was provided from certain Hopf algebras through 
{\it quantum heaps}.

For both algebraic and knot theoretic perspectives, it naturally arises the question of whether there exists a
(co)homology theory for algebras that has both associative and Yang-Baxter (YB) structures,
which unifies Hochschild and Yang-Baxter (co)homology theories.
Algebras with braiding, called braided algebras, have been defined in \cite{Baez}, where
a homology theory was also introduced.  
Homology theories of algebras with braiding present have been studied in contexts 
different from this paper, 
for example in  \cite{HK,KP}.
In this paper we propose a cohomology theory that unifies 
Yang-Baxter and Hochschild theories.
We take 
an approach from deformation theory, and formulate low dimensional 
differentials, and show that the cohomology groups play roles in integrability.
We point out that our approach differs from \cite{Baez} in that our aim is to obtain a unified
theory of Hochschild and YB cohomologies, while in \cite{Baez} the main objective was to
encode the braiding operation in the Hochschild homology by replacing the natural switching
morphism of the symmetric tensor categories of vector spaces (or modules) by the braiding.

 We hereby give a brief overview of the main results of this article. We introduce a chain complex (in low dimensions) that combines Hochschild's theory of algebra cohomology and Yang-Baxter cohomology, Proposition~\ref{pro:cochain}. This theory applies to braided algebras, answering in the affirmative the question regarding the existence of a cohomological theory unifying Hochschild cohomology and Yang-Baxter cohomology. We therefore show that the second cohomology group of this cochain complex classifies deformations of braided algebras, which is shown in Theorem~\ref{thm:classification}. This result fits in the usual paradigm that the second cohomology group of an algebraic structure with coefficients in itself, controls the deformation theory of the structure. In addition, we derive the obstruction for the existence of quadratic deformations in Theorem~\ref{lem:higher}, show that this lies in the third cohomology group in Lemma~\ref{lem:3cocy} and derive a sufficient condition for the existence of quadratic deformations in Corollary~\ref{cor:quadratic}.

The paper is organized as follows.
In Section~\ref{sec:ba}, definitions of braided algebras and examples will be presented,
for which the unified cohomology theories of YB and Hochschild can be
 applied.
Yang-Baxter Hochschild homology is defined up to 2-differentials  from deformation 
theory in Section~\ref{sec:deform}.
Relations to braided multiplications and Hopf algebra homology are discussed in Sections~\ref{sec:braidedmulti} and \ref{sec:Hopf_cohomology}, respectively. 
The theory is further extended to 3-differentials using higher deformations in Section~\ref{sec:3coh}.
Detailed proofs for quadratic deformations are deferred to the appendix.

	Throughout the article, following common conventions, we indicate cochain complexes by the letter $C$, cocycles by $Z$ and coboundaries by $B$, all of them with appropriate subscripts and superscripts to indicate the type of complex used, e.g. Hochschild or YB, and the dimension.

\section{Braided algebras} \label{sec:ba} 

In this section we define YI, IY, and braided algebras for which we define Yang-Baxter Hochschild (co)homology, and give examples of such algebras.
These are algebras with both multiplication and braiding that satisfy certain compatibility conditions. Throughout this section, and the rest of the article, we assume that all algebras are over a unital commutative ring $\mathbb k$,
and we use the symbol $\mathbb 1$ to indicate the identity map.
 The set ${\rm Hom}_{\mathbb k}( V , W)$ of $\mathbb k$-module homomorphisms, for $\mathbb k$-modules $V$ and $W$, is simply denoted by ${\rm Hom}(V , W)$.

Recall that for an invertible map $R: V \otimes V \rightarrow V \otimes V$
of a ${\mathbb k}$-module $V$,
the equation
$$(R \otimes {\mathbb 1}) ({\mathbb 1} \otimes R) (R \otimes {\mathbb 1}) 
=({\mathbb 1} \otimes R) (R \otimes {\mathbb 1}) ({\mathbb 1} \otimes R) $$
is called the {\it Yang-Baxter} equation (YBE), and a solution of it is called a Yang-Baxter (YB) operator.

\begin{definition}
{\rm
Let $(V, \mu)$ be a ${\mathbb k}$-algebra 
for a unital ring $\mathbb k$, with a YB operator $R: V \otimes V \rightarrow V \otimes V$.
We say that $\mu$ and $R$ satisfy the ${\rm YI}$  (resp. ${\rm IY}$) condition if they satisfy 
\begin{eqnarray}
(\mathbb 1\otimes \mu) ( R \otimes {\mathbb 1}) ({\mathbb 1} \otimes  R  ) &=& R(\mu\otimes {\mathbb 1})\label{eqn:YI}, \\ 
(\mu\otimes \mathbb 1) ({\mathbb 1} \otimes R )(R \otimes {\mathbb 1}) &=& R(\mathbb 1\otimes \mu) ,\label{eqn:IY}
\end{eqnarray}
respectively.
We call $(V, \mu, R)$ a YI algebra (resp. IY algebra),
the naming defined in \cite{JSC}.
 If $(V, \mu, R)$ satisfies both Equation~\eqref{eqn:YI} and Equation~\eqref{eqn:IY}, we call it a braided algebra.  
}
\end{definition}

\begin{definition}\label{def:braided_hom}
	{\rm 
	If $(V_1,\mu_1, R_1)$ and $(V_2,\mu_2, R_2)$ are braided algebras, we say that a linear map $f : V_1 \rightarrow V_2$ is a homomorphism of braided algebras, if $f$ is an algebra homomorphism such that the following diagram commutes
	\begin{center}
	\begin{tikzcd}
		V_1\otimes V_1\arrow[rr,"f\otimes f"]\arrow[d,"R_1"] & & V_2\otimes V_2\arrow[d,"R_2"]\\
		V_1\otimes V_1\arrow[rr,"f\otimes f"]& &V_2\otimes V_2 
		\end{tikzcd} 
	\end{center}
     If $f$ admits an inverse that is itself a braided algebra homomorphism, we say that $f$ is a braided isomorphism. 
 }
\end{definition}

A YB operator and the multiplication $\mu$ are depicted in Figure~\ref{IYYI} (A) and (B), respectively,
and using these diagrams, Equation~\eqref{eqn:YI} and Equation~\eqref{eqn:IY} are depicted in (C) and (D), respectively.
An  array of 
 vertical $n$ edges represents $V^{\otimes n}$, and the diagrams are read from top to bottom, in the direction of  homomorphisms.
 A crossing in (A) represents a YB operator $R: V^{\otimes 2} \rightarrow V^{\otimes 2}$, and 
a trivalent vertex  in (B) represents a multiplication $\mu: V ^{\otimes 2} \rightarrow V$.

\begin{figure}[htb]
\begin{center}
\includegraphics[width=5in]{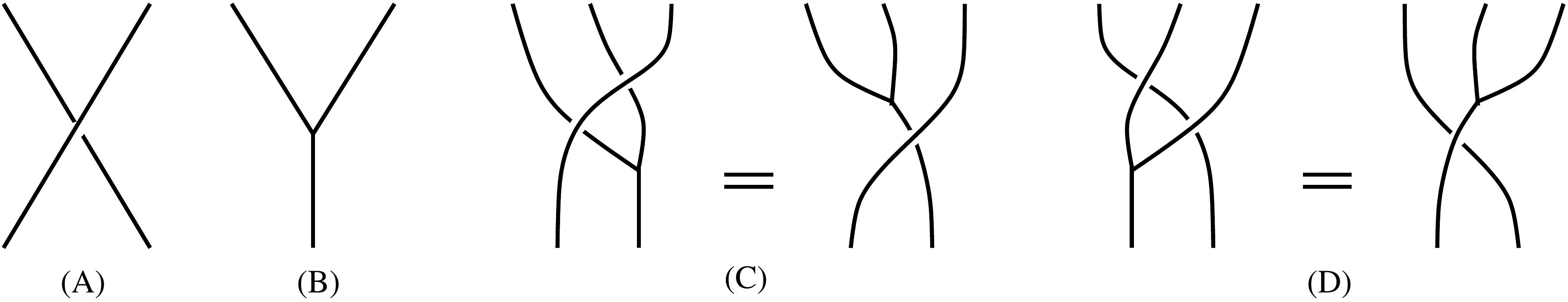}
\end{center}
\caption{}
\label{IYYI}
\end{figure}

\subsection{Braided algebras from tensorized multiple conjugation quandles}

We show that tensorization of multiple conjugation quandles (MCQs) \cite{CIST} produces 
braided algebras.

First, recall a \textit{quandle}~\cite{Joyce82,Matveev82}, is a non-empty set $X$ with a binary operation $*:X\times X\to X$ satisfying the following axioms.
\begin{itemize}
\item[(1)] For any $a\in X$, we have $a*a=a$.
\item[(2)] For any $a\in X$, the map $S_a:X\to X$ defined by $S_a(x)=x*a$ is a bijection.
\item[(3)] For any $a,b,c\in X$, we have $(a*b)*c=(a*c)*(b*c)$.
\end{itemize}
A {\it rack} is a set with an operation that satisfies (2) and (3).

\begin{definition}[\cite{Ishii08}]
	{\rm
A \textit{multiple conjugation quandle (MCQ)} $X$ is the disjoint union of groups $G_\lambda$,
where $\lambda$ is an element of a finite index set $\Lambda$, 
with a binary operation $*:X\times X\to X$ satisfying the following axioms.
\begin{itemize}
\item[(1)] For any $a,b\in G_\lambda$, we have $a*b=b^{-1}ab$.
\item[(2)] For any $x\in X$, $a,b\in G_\lambda$, we have $x*e_\lambda=x$ and $x*(ab)=(x*a)*b$, where $e_\lambda$ is the identity element of $G_\lambda$.
\item[(3)] For any $x,y,z\in X$, we have $(x*y)*z=(x*z)*(y*z)$ (self-distributivity).
\item[(4)] For any $x\in X$, $a,b\in G_\lambda$, 
we have $(ab)*x=(a*x)(b*x)$ in some group $G_\mu$.
\end{itemize}
}
\end{definition}

We call the group $G_\lambda$ a \textit{component} of the MCQ.
An MCQ is a type of quandle that can be decomposed as a union of groups, and the quandle operation in each component is given by conjugation.
Moreover, there are compatibilities, (2) and (4),  between the group and quandle operations.
In \cite{Ishii08}, concrete examples of MCQs are presented.

\begin{example}\label{ex:MCQ}
{\rm 
Let $X=\sqcup_{\lambda \in \Lambda} G_\lambda$ be an MCQ with a quandle operation $*$. 
Let ${\mathbb k}[X]$ be the free ${\mathbb k}$-module generated by $X$.
Define a multiplication $\mu : {\mathbb k}[X] \otimes {\mathbb k}[X] \rightarrow {\mathbb k}[X]$
on generators by $\mu (x\otimes y):= x y $ if $x, y \in G_\lambda$ for some $\lambda \in \Lambda$, 
and 
$\mu (x\otimes y) :=0$ otherwise, and extended linearly.
Define $R: {\mathbb k}[X]\otimes {\mathbb k}[X] \rightarrow {\mathbb k}[X] \otimes {\mathbb k}[X]$
on generators $x, y \in X$ by $R(x \otimes y):=  y \otimes (x*y)$ and extended linearly.
Then ${\mathbb k}[X]$ is a braided algebra, as we now proceed to show.
From the construction, it is seen that $R$ is a YBO.
If two of $x, y, z \in X$ belong to  distinct $G_\lambda$'s, then the triple product is $0$, 
and if all belong to the same $G_\lambda$ then it is associative, hence 
$\mu$ is associative.
Equations~\eqref{eqn:YI} and \eqref{eqn:IY}, respectively, follow from the conditions 
 $(ab)*x=(a*x)(b*x)$ and $x*(ab)=(x*a)*b$ 
 for  any $x\in X$, $a,b\in G_\lambda$ on generators, and if $a, b $ belong to distinct $G_\lambda$ then
 the values are $0$ for both sides of the equations, hence ${\mathbb k}[X]$ is a YI and IY algebra.
}
\end{example}

Another type of discrete racks with partial multiplication is a rack constructed from a {\it group heap}.
The following construction is found in  \cite{SZframedlinks, SZsfceribbon}.
Let $G$ be a group, and consider the ternary operation on $G$ defined by 
$T(x,y,z):=xy^{-1}z$. A group $G$ with this operation is called a heap, and $T$ satisfies the ternary 
self-distributive property,
$$T( T(x,y,z), u, v) =T( T(x, u,v), T(y,u,v), T(z,u,v)). $$
By this property, it follows that the binary operation defined on $X=G \times G$ by
$(x,y)* (u,v) := ( T(x,u,v), T(y, u, v))=(xu^{-1}v, yu^{-1}v)$ is self-distributive.
 Property (2) above is also satisfied, and $(X, *)$ is a rack.
An argument similar to the
 one in Example~\ref{ex:MCQ}
gives the following.

\begin{example}
	{\rm 
Let $G$ be a group heap, and 
let $(X, *)$, $X= G \times G$, be a rack defined above. 
Let ${\mathbb k}[X]$ be the free abelian group generated by $X$.
Define a multiplication $\mu : {\mathbb k}[X] \otimes {\mathbb k}[X] \rightarrow {\mathbb k}[X]$
on generators $(x,y), (u,v) \in X$ 
by $\mu ((x,y)\otimes (u,v)):=  (x,v)$ if $y=u$, and
$\mu (x\otimes y) :=0$ otherwise, and extended linearly.
Define $R: {\mathbb k}[X]\otimes {\mathbb k}[X] \rightarrow {\mathbb k}[X] \otimes {\mathbb k}[X]$
on generators $(x,y), (u,v) \in X$ by 
$$R((x,y) \otimes (u,v) ):=  (u, v) \otimes   [ (x,y)*(u,v) ] =
(u, v) \otimes (xu^{-1}v, yu^{-1}v)$$
 and extended linearly.
Then ${\mathbb k}[X]$ is  a braided algebra.
}
\end{example}

{
\subsection{Brief overview of Hopf algebras}

Before giving the constructions that constitute the main examples of braided algebras considered in this article, we briefly recall the notion of Hopf algebra. We also recall some of the properties of Hopf algebras needed to apply the results of \cite{SZbrfrob}, which we follow in the exposition of this subsection. 

A {\it Hopf algebra} $(X, \mu,  \eta,  \Delta, \epsilon, S)$ (a 
module over a unital ring 
$\mathbb k$,
multiplication, unit, comultiplication, counit, antipode,  respectively), is
defined as follows. 
First, recall that a bialgebra 
$X$  
is a module endowed with
a multiplication $\mu: X\otimes X\rightarrow X$ with unit $\eta$ and a comultiplication $\Delta: X\rightarrow X\otimes X$ with counit $\epsilon$ such that the compatibility conditions  
$$\Delta \circ  \mu = (\mu\otimes \mu)\circ ( {\mathbb 1}\otimes  \tau\otimes  {\mathbb 1}) \circ (\Delta\otimes \Delta),$$
$$ \epsilon \mu = \epsilon\otimes \epsilon,\ \ \ \Delta\eta = \eta\otimes \eta,\ \ \ \epsilon \eta = \mathbb 1
		$$
hold, where $\tau$ denotes the transposition $\tau(x \otimes y) = y\otimes x$ for simple tensors.  A Hopf algebra is a bialgebra endowed with a map $S: X\rightarrow X$, called {\it antipode}, satisfying the equations 
$$\mu \circ (\mathbb 1\otimes S)\circ \Delta = \eta \circ \epsilon = \mu\circ (S\otimes \mathbb 1)\circ\Delta,$$
 called the {\it antipode condition}. Antipodes are antihomomorphisms.

Any Hopf algebra satisfies the equality 
$S \circ \mu \circ \tau = \mu \circ (S \otimes S)$.
A Hopf algebra is called {\it involutory} if $S^2={\mathbb 1}$, the identity. 
It is known, \cite{Kas} Theorem~III.3.4, that if a Hopf algebra is commutative or cocommutative it follows that it is also involutory.

For the comultiplication, we use Sweedler's notation $\Delta(x)=x^{(1)}\otimes x^{(2)}$ suppressing the summation. Further, we use
$$( \Delta \otimes {\mathbb 1} ) \Delta (x) = ( x^{(11)} \otimes x^{(12)} ) \otimes x^{(2)}\quad {\rm and} \quad
( {\mathbb 1}  \otimes \Delta ) \Delta (x) =  x^{(1)} \otimes ( x^{(21)} \otimes x^{(22)} ) , $$
both of which are also written as 
$ x^{(1)} \otimes x^{(2)}  \otimes x^{(3)}$ from the coassociativity. 

A  left integral of $X$ is an element $\lambda\in X$ such that $x\lambda = \epsilon (x) \lambda$ for all $x\in X$,
where juxtaposition of elements denotes multiplication applied.
 Right integrals and (two-sided) integrals, are defined similarly.      
The existence of integrals is a fundamental tool to endow a Hopf algebra with a Frobenius structure. It is known that the set of integrals of a free finite dimensional 
Hopf algebra over a PID 
admits a one dimensional space of integrals, 
see \cite{LS}. More generally, a finitely generated projective Hopf algebra over a ring admits a left integral space of rank one \cite{Par}. Observe that when a Hopf algebra is (co)commutative, it follows that a left integral is also a right integral.
}

\subsection{ Braided algebras from Hopf algebras}\label{sec:Hopf_example}

We now describe an important class of braided algebras related to MCQs derived from Hopf algebras. First, recall \cite{EZ} that given a Hopf algebra $H$, one has a Yang-Baxter operator $R_H$ associated to it defined through the right adjoint action of $H$ on itself, which is a generalization of the notion of conjugation in a group:
$$
R_H: x\otimes y\mapsto y^{(1)}\otimes S(y^{(2)})xy^{(3)},
$$
where juxtaposition here indicates multiplication in $H$, and we have employed Sweedler's notation. Observe that when $H$ is the group algebra of a group $G$, where $S(g) = g^{-1}$ for all $g\in G$, then  the YB operator $R_H$ coincides with the YB operator obtained by linearizing the group $G$ and endowing the corresponding group algebra with the standard Hopf algebra structure, where the comultiplication is diagonal and the antipode is induced by taking inverses in $G$. 

\begin{lemma}\label{lem:BA}
	Let $(H,\mu,\Delta,\eta,\epsilon. S)$ be a Hopf algebra and let $R_H$ denote the YB operator defined above. Then $(H, \mu, R_H)$ is a braided algebra. 
\end{lemma}
\begin{proof}
	We show that Equation~\eqref{eqn:IY} 
	holds. 
	A similar procedure can be performed for Equation~\eqref{eqn:YI}. 
	For the left hand side of Equation~\eqref{eqn:IY} 
	we have
	\begin{eqnarray*}
	(\mu\otimes \mathbb 1)(\mathbb 1\otimes R_H)(R_H\otimes \mathbb 1)(x\otimes y\otimes z) &=&
	(\mu\otimes \mathbb 1)(\mathbb 1\otimes R_H)(y^{(1)}\otimes S(y^{(2)})xy^{(3)}\otimes z)\\
	&=& (\mu\otimes \mathbb 1)(y^{(1)}\otimes z^{(1)}\otimes S(z^{(2)})S(y^{(2)})xy^{(3)}z^{(3)})\\
	&=& y^{(1)}z^{(1)}\otimes S(z^{(2)})S(y^{(2)})xy^{(3)}z^{(3)}
	\end{eqnarray*}
while for the right hand side of Equation~\eqref{eqn:IY} 
we have
\begin{eqnarray*}
R_H(\mathbb 1\otimes \mu)(x\otimes y\otimes z) &=& R_H(x\otimes yz)\\
&=& (yz)^{(1)}\otimes S((yz)^{(2)})x(yz)^{(3)},
\end{eqnarray*}
 The two expressions are easily seen to coincide, using the fact that $\mu$ and $\Delta$ satisfy the Hopf algebra axioms, and the fact that $S$ is a antihomomorphism of algebras. 
\end{proof}

\subsection{Braided Frobenius algebras from (co-)commutative Hopf algebras}

In \cite{SZbrfrob},  {\it braided Frobenius algebras} (i.e. a class of braided algebras) were constructed from 
   commutative and cocommutative Hopf algebras. 
   This is a Hopf algebra version of the heap construction described above.
   We briefly review the construction.

Let $(X, m,\eta,\Delta,\epsilon,S)$ be a commutative and cocommutative Hopf algebra,
where the listed data are multiplication, unit, comultiplication, counit, and the antipode,
respectively. 
Then $V=X \otimes X$ was given a braided Frobenius algebra structure, in particular braided algebra, 
with the following YB operator and the multiplication. 

The YB operator is defined as follows. First, introduce $T: X\otimes X\otimes X\rightarrow X$ by letting $T(x \otimes y \otimes z) $ be defined on simple tensors as
$T(x \otimes y \otimes z) :=x S(y) z$, where juxtaposition indicates, as before, Hopf algebra multiplication $m$. Note that this corresponds to the operation $xy^{-1}z$ for  group heaps.
Then the map $R: V^{\otimes 2} \rightarrow  V^{\otimes 2}$ is defined for 
simple tensors
 by 
$$ R ( ( x\otimes y)  \otimes ( z \otimes w)  ) := 
( z^{(1)} \otimes w^{(1)} ) \otimes 
T(x \otimes z^{(2)} \otimes w^{(2)} ) \otimes
T(y  \otimes z^{(3)} \otimes w^{(3)} ) .$$
It was shown that this $R$ indeed satisfies the YBE.

The multiplication $\mu$ on $V \otimes V$ was defined as follows.
There exists an integral $\lambda: V \rightarrow {\mathbb k}$ in a Hopf algebra in question, and 
the map $\cup: V \otimes V \rightarrow {\mathbb k}$ is defined by
$\cup:=\lambda m ( {\mathbb 1} \otimes S)$. 
Then the multiplication is defined by $\mu = {\mathbb 1} \otimes \cup \otimes {\mathbb 1} $.
It was shown in \cite{SZbrfrob} that these $R$ and multiplication give rise to braided algebras.

In summary, in this section,  definitions of YI condition (Equation~\eqref{eqn:YI}), IY condition (Equation~\eqref{eqn:IY}), braided algebras, and concrete examples of such structures
 have been
presented. In the following sections we propose (co)homology theories for these structures 
that unify both Hochschild and Yang-Baxter homology theories.

\section{Yang-Baxter Hochschild cohomology up to dimension 3 and deformations} \label{sec:deform}

In this section we propose low dimensional Yang-Baxter Hochschild cohomology from point of view of deformation theory, and prove the classification theorem for deformations up to equivalence. The main result of this section is Theorem~\ref{thm:classification}, which motivates the definition of Yang-Baxter Hochschild chomology of this article, as the algebraic tool controlling infinitesimal deformations.

\subsection{Hochschild cohomology}

In this section we briefly review deformation theoretic aspects of low dimensional 
Hochschild cohomology for the simplest case.
Let $(V, \mu)$ be an associative algebra over ${\mathbb k}$.
All homomorphism groups,  ${\rm Hom}_{\mathbb k}$ over ${\mathbb k}$, are denoted by 
${\rm Hom}$ for simplicity.
The cochain groups of Hochschild cohomology are defined to be  $C_{\rm H}^n(V,V)={\rm Hom}(V^{\otimes n},V)$ for $n\geq 0$. 

 The differential in degree zero is given by
	$$
	\delta^0_{\rm H}(s) = \mu(\mathbb 1\otimes s) - \mu(s\otimes \mathbb 1),
	$$
	for $s : \mathbb k \rightarrow V$.

The differentials $ \delta^1_{\rm H}$ and $\delta^2_{\rm H}$ are defined for $f \in C_{\rm H}^1(V,V)$ and $\psi \in C_{\rm H}^2(V,V)$ by 
\begin{eqnarray*}
\delta^1_{\rm H}(f)&=&
  \mu ( f \otimes {\mathbb 1}) + \mu  ( {\mathbb 1} \otimes  f ) 
  - f \mu , \\
\delta^2_{\rm H}(\psi)&=&  
\mu (\psi \otimes {\mathbb 1}) + \psi (\mu \otimes {\mathbb 1})
- \mu ( {\mathbb 1}\otimes  \psi ) - \psi ({\mathbb 1}\otimes   \mu ) .
\end{eqnarray*}
		We note that usual differential for Hochschild cohomology has the opposite sign of $\delta^2_{\rm H}$ written above. 
		This convention is for convenience of computations of deformations and diagrams that follow, and also similar to \cite{CCES-coalgebra}. 
		This does not have any implication in the results that follow.

Diagrammatic presentations of the differentials are depicted in Figures~\ref{Hochdiff1} and \ref{Hochdiff2}. 
An  array of 
 vertical $n$ edges represents $V^{\otimes n}$, and a  circle on an edge represents $f: V \rightarrow V$.
The diagrams are read from top to bottom, in the direction of a homomorphism.
A trivalent vertex  represents a multiplication $\mu: V ^{\otimes 2} \rightarrow V$, 
and a circled trivalent vertex represent $\psi$. 

\begin{figure}[htb]
\begin{center}
\includegraphics[width=1.5in]{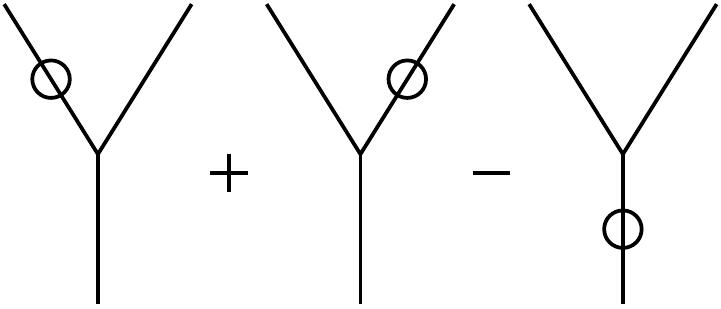}
\end{center}
\caption{}
\label{Hochdiff1}
\end{figure}

\begin{figure}[htb]
\begin{center}
\includegraphics[width=2.5in]{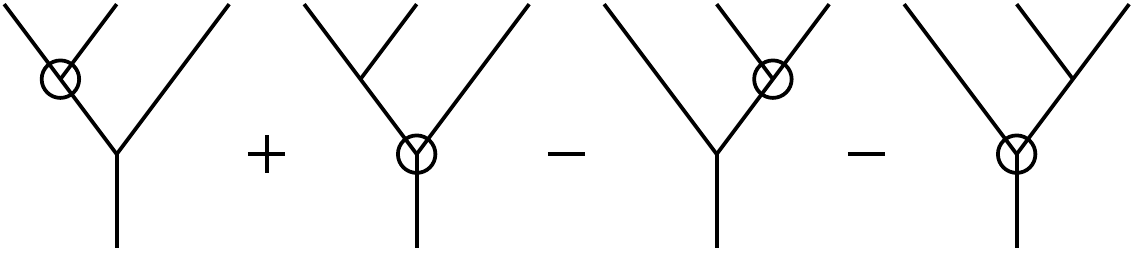}
\end{center}
\caption{}
\label{Hochdiff2}
\end{figure}

	Let $(V,\mu)$ be an algebra over $\mathbb k$.  
We say that an  algebra $(V',\mu')$ over the power series $\mathbb k[[\hbar]]$
 is a deformation of $V$ if the quotient algebra $(V' / (\hbar) V' , \mu')$ coincides with 
 $(V, \mu)$. 

	Let $(V,\mu)$ be an associative algebra, 
	$\tilde V=V \otimes_{\mathbb k} {\mathbb k}[[\hbar]]/(\hbar^2 ) \cong V \oplus \hbar V$,
	and  $\psi \in Z^2_{\rm H}(V,V)$. 
 Then, setting $\tilde \mu = \mu + \hbar \psi$,
  $(\tilde V,\tilde \mu)$ is an associative algebra if and only if the equations
	$\delta^2_{\rm H}( \psi)=0$ hold. 
 This is seen by computing the associativity  in $\tilde V$,
\begin{eqnarray*}
\tilde \mu (\tilde \mu \otimes {\mathbb 1} )& =&  
( \mu + \hbar \psi ) ( (\mu  + \hbar \psi  ) \otimes {\mathbb 1} ) \\
&=&
 \mu ( \mu \otimes {\mathbb 1} ) + \hbar [ \mu ( \psi   \otimes {\mathbb 1} )+
 \psi ( \mu  \otimes {\mathbb 1} ) ] , \\ 
\tilde \mu ({\mathbb 1}  \otimes \tilde \mu ) &=& 
 ( \mu + \hbar \psi ) ( {\mathbb 1}  \otimes  (\mu  + \hbar \psi  )) \\
 &=&
 \mu ({\mathbb 1}  \otimes  \mu ) + \hbar[ 
 \mu (  {\mathbb 1}   \otimes  \psi)+
  \psi (  {\mathbb 1}  \otimes \mu ) ] . 
  \end{eqnarray*}

In this case, we also have $\tilde \mu \equiv_{(\hbar)} \mu$, meaning that $\tilde \mu$ coincides with $\mu$ modulo $(\hbar)$, and we say that $(\tilde V, \tilde \mu)$ is an infinitesimal deformation   
of $(V, \mu)$. 
 Thus we say  that the primary obstruction to the {\it infinitesimal} deformation vanishes if and only if $\delta^2_{\rm H}(\psi)=0$.

\subsection{Yang-Baxter cohomology up to dimension 3}

In this section we define Yang-Baxter (deformation) cohomology.
Let $(V, R)$ be a ${\mathbb k}$-module with the YB operator $R : V^{\otimes 2} \rightarrow V^{\otimes 2}$.
The cochain groups are defined by 
$C^0_{\rm YB}(V,V)=0$ and 
$C^n_{\rm YB}(V,V)={\rm Hom} (V^{\otimes n},V^{\otimes n})$ for $n>0$.
We define the differentials for $f \in C^1_{\rm YB}(V,V) $
and $\phi \in C^2_{\rm YB}(V,V)$ by 
  \begin{eqnarray*}
 \delta^1_{\rm YB} (f) 
 &=&
 R ( f \otimes {\mathbb 1}) + R ( {\mathbb 1} \otimes  f ) 
 -  ( f \otimes {\mathbb 1}) R - ( {\mathbb 1} \otimes  f )  R ,\\
\delta^2_{\rm YB} (\phi) &=&
(R \otimes {\mathbb 1} ) ( {\mathbb 1}  \otimes R ) ( \phi \otimes  {\mathbb 1} )
+ (R \otimes {\mathbb 1} ) ( {\mathbb 1}  \otimes \phi ) ( R \otimes  {\mathbb 1} )
+  (\phi \otimes {\mathbb 1} ) ( {\mathbb 1}  \otimes R ) ( R \otimes  {\mathbb 1} ) \\
& & -  ( {\mathbb 1}  \otimes R ) ( R \otimes  {\mathbb 1} ) ( {\mathbb 1}  \otimes \phi ) 
- ( {\mathbb 1}  \otimes R ) ( \phi \otimes  {\mathbb 1} ) ( {\mathbb 1}  \otimes  R ) 
- ( {\mathbb 1}  \otimes \phi ) ( R \otimes  {\mathbb 1} ) ( {\mathbb 1}  \otimes  R ) .
\end{eqnarray*}
The differentials are depicted in Figures~\ref{YBdiff1} and \ref{YBdiff2}.
The cochains $f$ and $\phi$ are represented by circles on an edge and a crossing, 
respectively. 

\begin{figure}[htb]
\begin{center}
\includegraphics[width=2.2in]{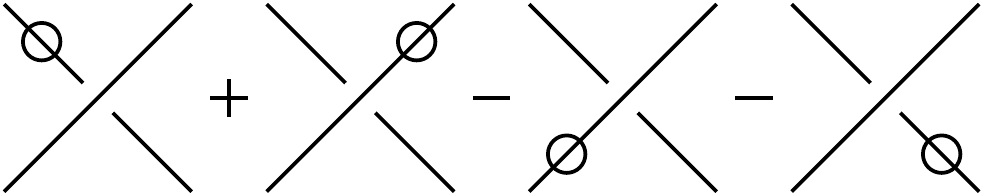}
\end{center}
\caption{}
\label{YBdiff1}
\end{figure}

\begin{figure}[htb]
\begin{center}
\includegraphics[width=3.5in]{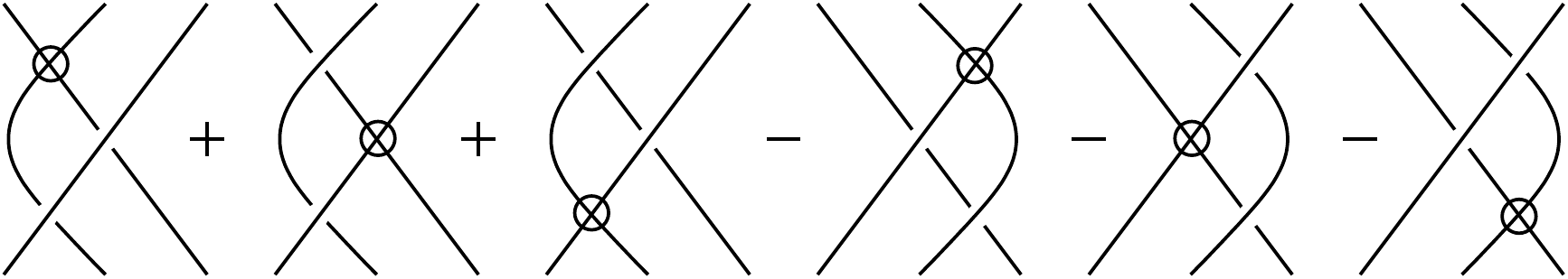}
\end{center}
\caption{}
\label{YBdiff2}
\end{figure}

\begin{lemma} \label{lem:R}
The sequence
$C^0_{\rm YB}(V,V) \rightarrow 
C^1_{\rm YB}(V,V) \stackrel{\delta^1_{\rm YB}}{\longrightarrow}
C^2_{\rm YB}(V,V) \stackrel{\delta^2_{\rm YB}}{\longrightarrow}
C^2_{\rm YB}(V,V)
$
defines a cochain complex.
\end{lemma}

\begin{proof}
Let $f$ be a YB $1$-cochain. Then, the lemma
is proved by showing that $\delta^2_{\rm YB}(\phi)=0$ when $\phi=\delta^1_{\rm YB}(f)$. 
Diagrammatically, when $\delta^1_{\rm YB}(f)$ is substituted in $\phi$, 
the circles crossing representing $\phi$ are replaced by four terms with circles placed at the four edges 
adjacent to the crossing. For example, the first term of $\delta^2_{\rm YB}(\phi)$ is shown in 
Figure~\ref{YBdiff1diff2} left top, and the replacement representing the substitution
is depicted in the right of the top row. 
Similar diagrams are shown in the middle row for the second term of $\delta^2_{\rm YB}(\phi)$,
and  the bottom row in Figure 6 comes from the fifth summand in Figure 5. 
One sees that the terms labeled (B) and (C) in the top and the middle row cancel with opposite signs. 
This cancelation principle is explained by the fact that each edge is shared by two crossings.
In Figure~\ref{YBdiff1diff2}, the terms labeled (A) in the top row and (D) in the bottom row
have  the circles for $f$  placed at the end edge (top).
They are canceled by opposite signs after applying the YBE (in diagrammatic form).
Hence in both cases (the circle placed inside or end edges), 
two terms cancel with opposite signs, and we obtain  $\delta^2_{\rm YB}(\phi)=0$.
\end{proof}

\begin{figure}[htb]
\begin{center}
\includegraphics[width=3in]{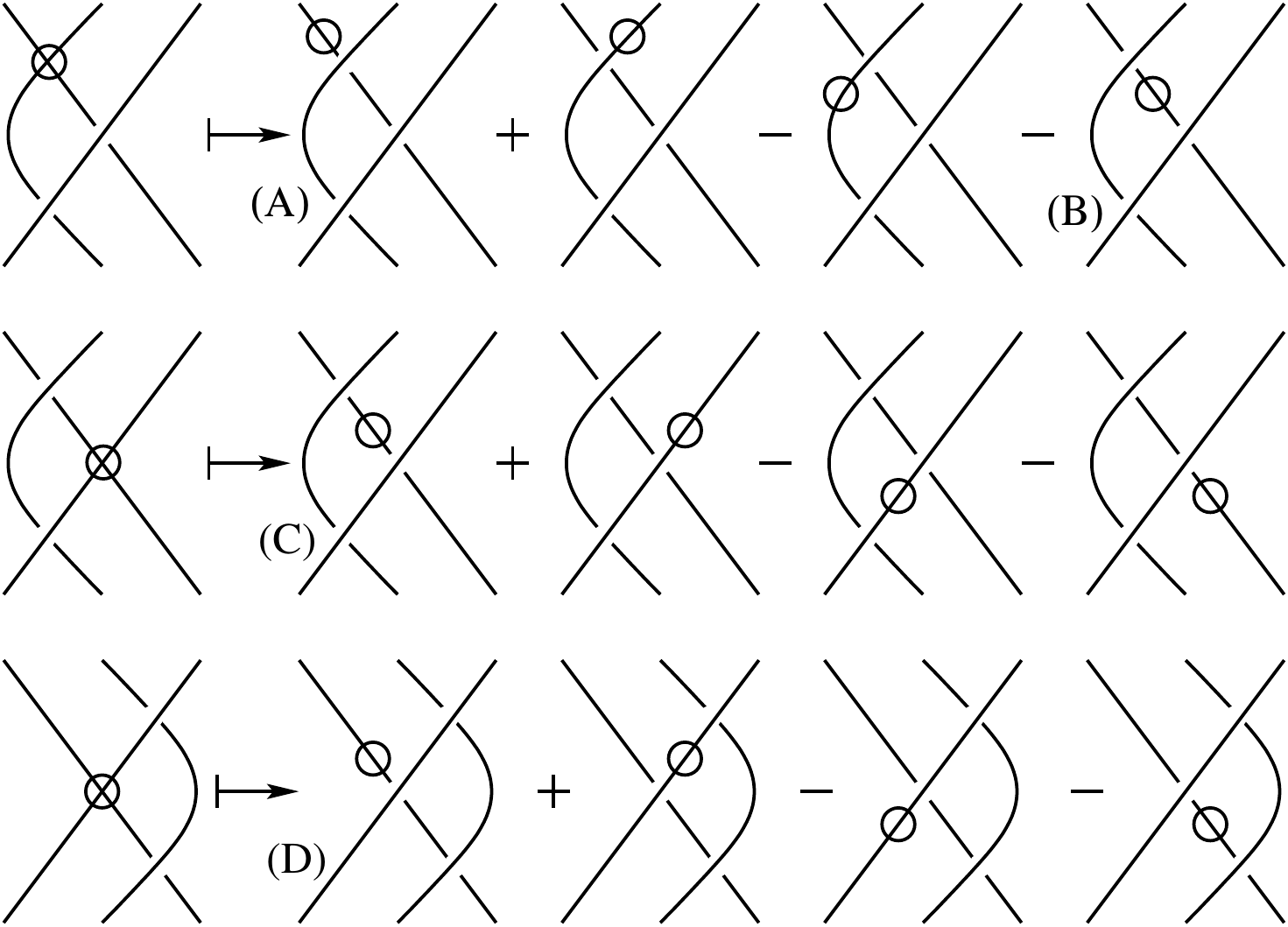}
\end{center}
\caption{}
\label{YBdiff1diff2}
\end{figure}

\begin{remark} 
{\rm

A deformation cohomology theory for Yang-Baxter operators was developed further to all higher dimensions  in \cite{SZ-BC}
extending those in this paper, using the diagrammatic method  developed above.

This YB cohomology theory    is different from those  in  \cite{Eisermann};
the  differentials in  \cite{Eisermann}
differ from ours by compositions of YB operators.
For example, the 2-cochain $c'$ in \cite{Eisermann} is written as $Rc$ for a 2-cochain $c$ in the cohomology considered in this article.
We adopt our definition and diagrammatics 
along the line of \cite{CCES-coalgebra,CCES-adjoint}
for the purpose of unifying Yang-Baxter cohomology with algebra cohomology.

}
\end{remark}

%
%
%

\subsection{Yang-Baxter Hochschild cohomology up to dimension $3$}

Let $V$ be a braided algebra with coefficient unital ring ${\mathbb k}$.
We define the cochain groups for a braided algebra $V$ with coefficients in itself up to degree $3$ as follows. 
We set $C_{\rm YBH}^0(V,V) =0$,
and  
$C^{n,k}_{\rm YBH}(V,V) = {\rm Hom}(V^{\otimes n},V^{\otimes k}) $
for $n , k >0$.
We also use different subscripts 
$$C^{n,k}_{\rm YI}(V,V) =  {\rm Hom}(V^{\otimes n},V^{\otimes k})
= C^{n,k}_{\rm IY}(V,V)$$
to distinguish different isomorphic direct summands. 
Define  
\begin{eqnarray*}
C_{\rm YBH}^1(V,V) &=& C^{1,1}_{\rm YBH}(V,V) ={\rm Hom}(V,V) , \\
  C^{2}_{\rm YBH}(V,V) &=& C^{2,2}_{\rm YBH}(V,V) \oplus C^{2,1}_{\rm YBH}(V,V), \quad {\rm and} \\
C^{3}_{\rm YBH}(V,V)  &=& C^{3,3}_{\rm YBH}(V,V)  \oplus
 C^{3,2}_{\rm YI}(V,V)  \oplus C^{3,2}_{\rm IY}(V,V) 
 \oplus  
 C^{3,1}_{\rm YBH}(V,V) . 
 \end{eqnarray*}

 We define differentials as follows.
 We set $\delta^1_{\rm YBH}$ to be the direct sum $\delta^1_{\rm YB}\oplus \delta^1_{\rm H}$.

   The second differential $\delta^2_{\rm YBH}$ is defined to be the direct sum of four terms $\delta^2_{\rm YB}\oplus \delta^2_{\rm YI}\oplus \delta^2_{\rm IY}\oplus \delta^2_{\rm H}$ where $\delta^2_{\rm YB}$ and $\delta^2_{\rm H}$ map in the first ($C^{3,3}_{\rm YBH}(V,V) ={\rm Hom}(V^{\otimes 3},V^{\otimes 3})$) 
  and last $(C^{3,1}_{\rm YBH}(V,V) = {\rm Hom}(V^{\otimes 3},V)$) direct summands of $C^3_{\rm YBH}(V,V)$, respectively,
  while $\delta^2_{\rm YI}$ and $\delta^2_{\rm IY}$ map to the middle  
two direct summands $C^{3,2}_{\rm YI}(V,V)$ and $C^{3,2}_{\rm IY}(V,V)$, respectively.
  Each differential is defined as follows.
  \begin{eqnarray*}
\delta^2_{\rm YI} (\phi\oplus \psi) &=&
 (\mathbb 1\otimes \psi)(R\otimes \mathbb 1)(\mathbb 1 \otimes R) + (\mathbb 1 \otimes \mu)(\phi\otimes \mathbb 1)(\mathbb 1\otimes R) \\
 & &
 + (\mathbb 1 \otimes \mu)(R\otimes \mathbb 1 )(\mathbb 1 \otimes \phi) 
-  R (\psi\otimes \mathbb 1 ) - \phi (\mu\otimes \mathbb 1). \\
\delta^2_{\rm IY}(\phi\oplus \psi) &=& 
(\psi\otimes \mathbb 1)(\mathbb 1\otimes R)(R\otimes \mathbb 1) + (\mu\otimes \mathbb 1)(\mathbb 1\otimes \phi)(R\otimes \mathbb 1) \\
& &  
+ (\mu\otimes \mathbb 1)(\mathbb 1 \otimes R)(\phi\otimes \mathbb 1)
- R(\mathbb 1 \otimes \psi)  - \phi (\mathbb 1\otimes \mu) .
\end{eqnarray*}
The differential $\delta^2_{\rm IY}(\phi\oplus \psi)$ is represented diagrammatically in Figure~\ref{IY2cocy}, where 2-cochains $\phi \in C^{2,2}_{\rm YBH}(V,V)$ and $\psi\in C^{3,1}_{\rm YBH}(V,V)$ are represented by 4-valent (resp. 3-valent) vertices 
with circles. The YI and IY components of the cochain complex above are included to enforce the coherence axioms between deformed algebra structure, and deformed YB operator. In other words, they ensure that YBH $2$-cocycles satisfy Equation~\eqref{eqn:YI} and Equation~\eqref{eqn:IY}, respectively.

\begin{figure}[htb]
\begin{center}
\includegraphics[width=3in]{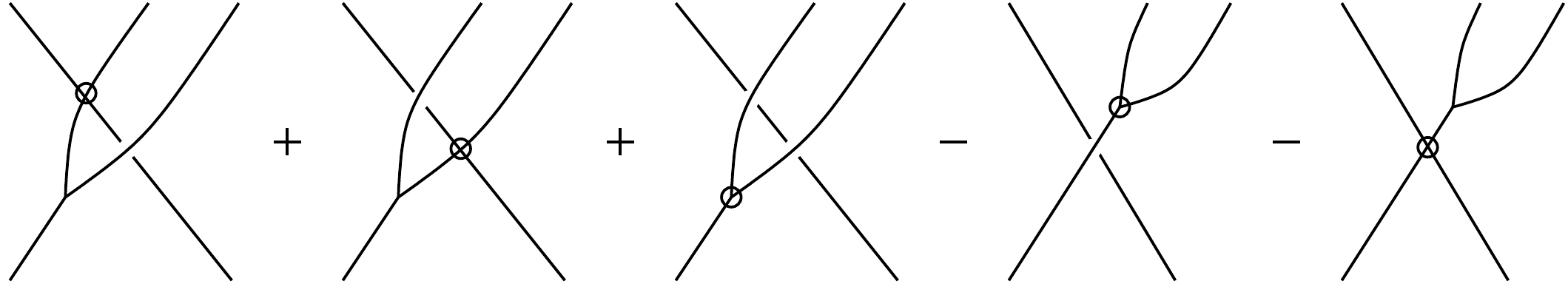}
\end{center}
\caption{}
\label{IY2cocy}
\end{figure}

\begin{proposition}\label{pro:cochain}
The sequence
$$ C_{\rm YBH}^0(V,V) \rightarrow C_{\rm YBH}^1(V,V) 
\stackrel{\delta^1_{\rm YBH}}{\longrightarrow}
C_{\rm YBH}^2(V,V) 
\stackrel{\delta^2_{\rm YBH}}{\longrightarrow}
 C_{\rm YBH}^3(V,V) 
 $$
 defines a cochain complex.
\end{proposition}

\begin{proof}
Let $f$ be a YBH $1$-cochain.
Setting $\psi=\delta^1_{\rm H}(f)$ and $\phi=\delta^1_{\rm YB}(f)$, we show that 
$\delta^2_i ( \phi \oplus \psi ) =0$ for $i={\rm YI}, {\rm IY}$.
The procedure showing that  $\delta^2_{\rm IY} ( \phi  \oplus  \psi )=0$  is depicted in Figures~\ref{IYder1} and~\ref{IYder2}, and similar to the proof of Lemma~\ref{lem:R}. 
In the right column of Figure~\ref{IYder1}, the positive terms of  $\delta^2_{\rm  IY} ( \phi \oplus
\psi )$ depicted in Figure~\ref{IY2cocy} are shown.
Circled vertices represent $\psi$ (3-valent vertices) and $\phi $ (4-valent ones). 
Under the assumption $\psi=\delta^1_{\rm H}(f)$ and $\phi=\delta^1_{\rm YB}(f)$,
each vertex is replaced by a linear combination of the diagrams with circles (representing $f$) on edges, according to the definition of $\delta^1(f)$ (Figure~\ref{Hochdiff1}). 
Such linear combinations for each term in the left column of Figure~\ref{IYder1}
are depicted in the right column of the figure. Thus $\psi=\delta^1_{\rm H}(f)$ and $\phi=\delta^1_{\rm YB}(f)$ is represented by replacing the positive terms in Figure~\ref{IY2cocy} by the right columns 
in Figure~\ref{IYder1}. The negative terms are similarly replaced by 
the right columns in Figure~\ref{IYder2}. With signs one sees that all terms cancel,
indicating $\delta^2_{\rm IY} ( \phi  \oplus  \psi)=0$.
The case $\delta^2_{\rm YI} ( \phi \oplus  \psi)=0$ is similar.
The other differentials for YBE and Hochschild are similar as well.
\end{proof}

\begin{figure}[htb]
\begin{center}
\includegraphics[width=3.2in]{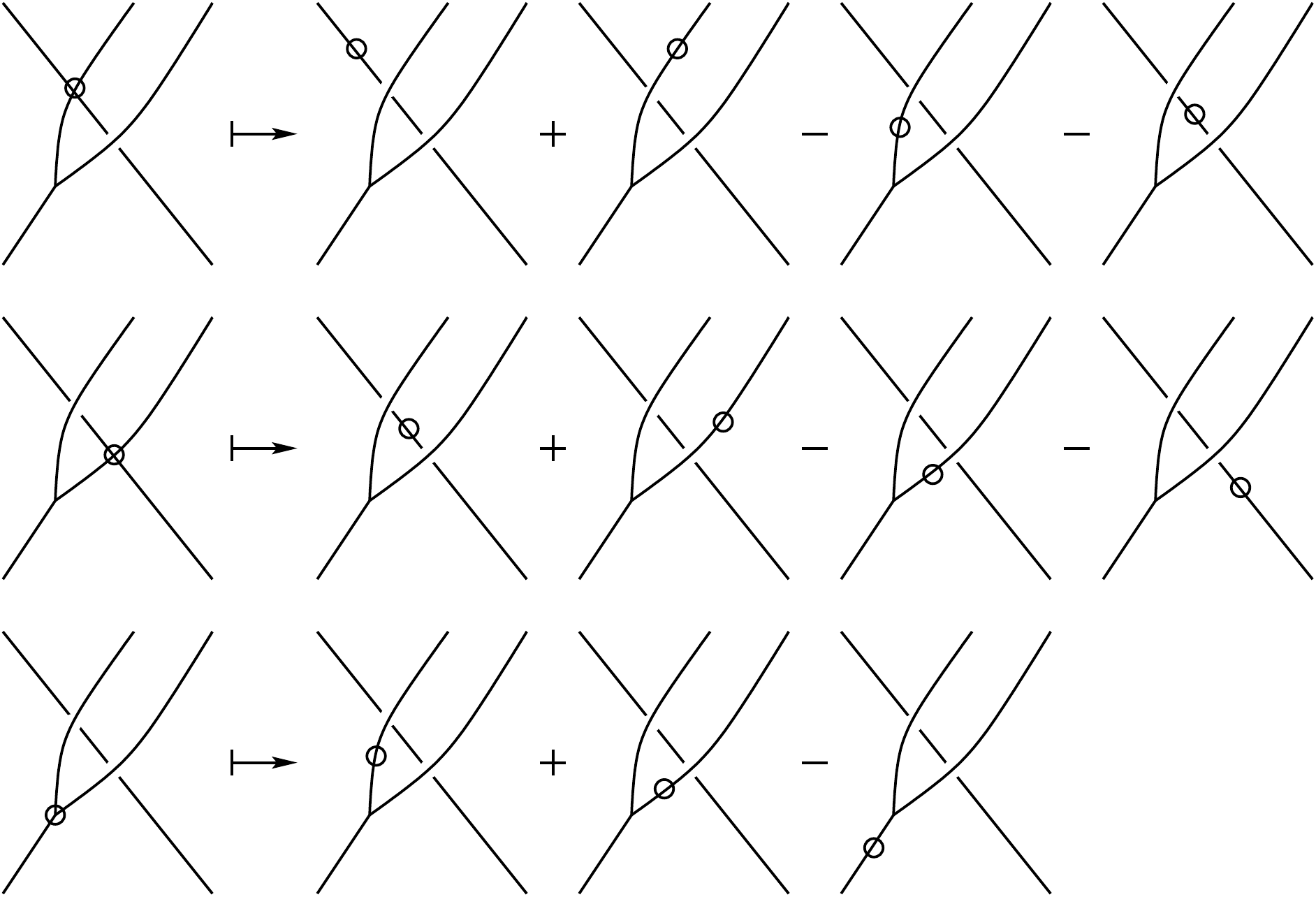}
\end{center}
\caption{}
\label{IYder1}
\end{figure}

\begin{figure}[htb]
\begin{center}
\includegraphics[width=3.5in]{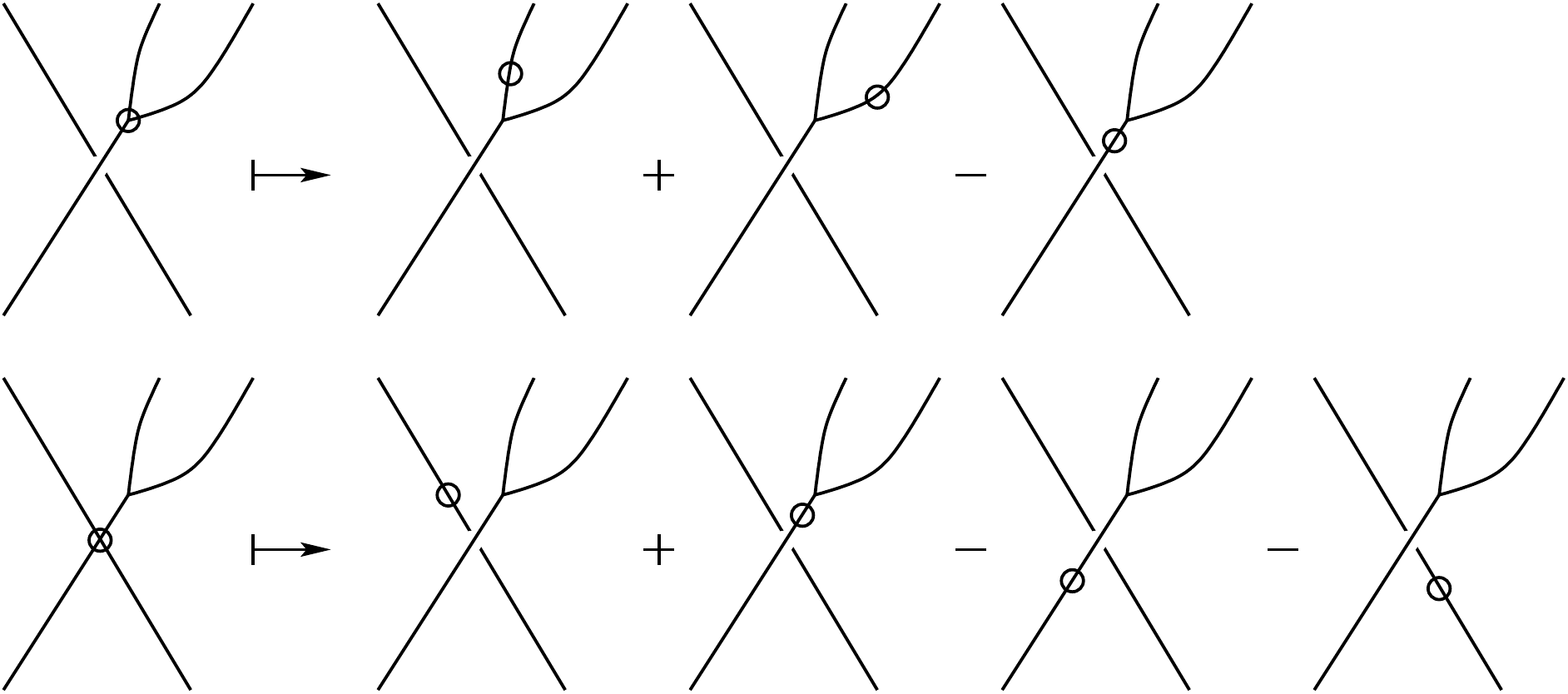}
\end{center}
\caption{}
\label{IYder2}
\end{figure}

\begin{remark}
{\rm
The YBH cohomology is further generalized in \cite{SZ-BC} to all higher dimensions
in a multicomplex form, that extends the definition for low dimensions presented in this paper.

}
\end{remark}

\subsection{Infinitesimal deformations of braided algebras}\label{sec:infinitesimal_deform}

	Let $(V,\mu,R)$ be a braided algebra over $\mathbb k$.  Let us consider the power series ring $\mathbb k[[\hbar]]$ over the formal variable $\hbar$, and let $(\hbar )$ denote the 
	ideal of $\mathbb k[[\hbar]]$ generated by $\hbar$.



\begin{definition} \label{def:inf_deform} 
	{\rm 

	Let $(V, \mu, R)$ be a braided algebra over ${\mathbb k}$. Extend $\mu$ and $R$ 
	to $\tilde V=V \otimes_{\mathbb k} {\mathbb k}[[\hbar]] / (\hbar^2 ) \cong V \oplus \hbar V$ by linearly extending on $\hbar V$  and use the same notation $\mu$ and $R$ on $\tilde V$. 	We say that $(\tilde V, \tilde \mu , \tilde R)$ is an {\it infinitesimal deformation} of $(V, \mu, R)$ 
	if ${\rm Im }(\tilde \mu - \mu) \subset \hbar  V$ and 
	${\rm Im }(\tilde R - R ) \subset \hbar  ( V \otimes  V)$.
Similar definitions hold for YI algebras and IY algebras. 

	Let $(V, \mu, R)$ be a braided algebra, and let $(\tilde V, \tilde \mu, \tilde R)$ and $(\hat V, \hat \mu, \hat R)$
	be two infinitesimal deformations of $V$. Let $\tilde f: \tilde V \rightarrow \hat V$ be a homomorphism of braided
	algebras (Definition~\ref{def:braided_hom}). Then, we say that $f$ is a {\it homomorphism of infinitesimal deformations} if $\tilde f |_{V}  : V \rightarrow V$
	is the identity map $\mathbb 1_V$. A homomorphism of infinitesimal deformations that is invertible through a
	homomorphism of infinitesimal deformations is said to be an {\it isomorphism of infinitesimal deformations}. 

}
\end{definition}

If $(\tilde V, \tilde \mu , \tilde R)$ is an infinitesimal deformation of $(V, \mu, R)$, then we can write
\begin{eqnarray*}
			\tilde \mu &=& \mu + \hbar \psi\\
			\tilde R  &=& R + \hbar \phi,
		\end{eqnarray*}
		where $\psi : V\otimes V\rightarrow V$ and $\phi: V\otimes V \rightarrow V\otimes V$.

\begin{remark}
	{\rm 
		Recall that in order for $\tilde \mu$ to satisfy the associativity condition, $\psi$ needs to be a Hochschild $2$-cocycle, while $\tilde R$ is a YB operator if and only if $\phi$ is a YB $2$-cocycle. 	
		
		Since the latter fact is not found as commonly in the literature as the former, we give a brief proof of it. See also \cite{Eisermann,Eisermann1}. We set $\phi_0 := R$ and $\phi_1 := \phi$ for notational convenience. The YBE for $\tilde R$ then gives us for the left hand side (modulo terms of quadratic order or higher in $\hbar$)
		\begin{eqnarray*}
				(\tilde R\otimes \mathbb 1)(\mathbb 1\otimes \tilde R)(\tilde R\otimes \mathbb 1) &=& (\phi_0\otimes \mathbb 1)(\mathbb 1\otimes \phi_0)(\phi_0\otimes \mathbb 1) + \hbar \sum_{i+j+k=1} (\phi_i\otimes \mathbb 1)(\mathbb 1\otimes \phi_j)(\phi_k\otimes \mathbb 1).
		\end{eqnarray*}
		Similarly, for the right hand side of the YBE, modulo terms at least quadratic in $\hbar$ we have
		\begin{eqnarray*}
				(\mathbb 1\otimes \tilde R)(\tilde R\otimes \mathbb 1)(\mathbb 1\otimes \tilde R) &=& (\mathbb 1\otimes \phi_0)(\phi_0\otimes \mathbb 1)(\mathbb 1\otimes \phi_0) + \hbar \sum_{i+j+k=1} (\mathbb 1\otimes \phi_i)(\phi_j\otimes \mathbb 1)(\mathbb 1\otimes \phi_k).
		\end{eqnarray*}
		Therefore, the YBE holds for $\tilde R$ if and only if
		\begin{eqnarray*}
			0 &=&	\sum_{i+j+k=1} (\phi_i\otimes \mathbb 1)(\mathbb 1\otimes \phi_j)(\phi_k\otimes \mathbb 1) - \sum_{i+j+k=1} (\mathbb 1\otimes \phi_i)(\phi_j\otimes \mathbb 1)(\mathbb 1\otimes \phi_k)\\
			&=& (\phi_1\otimes \mathbb 1)(\mathbb 1\otimes \phi_0)(\phi_0\otimes \mathbb 1) + (\phi_0\otimes \mathbb 1)(\mathbb 1\otimes \phi_1)(\phi_0\otimes \mathbb 1) + (\phi_0\otimes \mathbb 1)(\mathbb 1\otimes \phi_0)(\phi_1\otimes \mathbb 1) \\
			&& - (\mathbb 1\otimes \phi_1)(\phi_0\otimes \mathbb 1)(\mathbb 1\otimes \phi_0) - (\mathbb 1\otimes \phi_0)(\phi_1\otimes \mathbb 1)(\mathbb 1\otimes \phi_0) - (\mathbb 1\otimes \phi_0)(\phi_0\otimes \mathbb 1)(\mathbb 1\otimes \phi_1) 
		\end{eqnarray*}
		which is exactly the $2$-cocycle condition for YB cohomology.
	}
\end{remark}

\begin{lemma}\label{lem:IYIdeform}
	Let $(V,\mu,R)$ be a braided algebra, and let $\psi$ and $\phi$ denote an algebra $2$-cocycle and a YB $2$-cocycle, respectively. Then, setting $\tilde \mu = \mu + \hbar \psi$ and $\tilde R = R + \hbar \phi$, $(\tilde V,\tilde \mu, \tilde R)$ is an infinitesimal braided algebra deformation of $(V,\mu,R)$ if and only if 
	$\delta^2_{\rm YBH}(\phi\oplus \psi)=0$ hold.
\end{lemma}
\begin{proof}
	Since $\psi$ and $\phi$ are algebra and YB $2$-cocycles, respectively, it follows that $\tilde \mu$ 
	is associative, and that $\tilde R$ is a YB operator. Now, we need to show that $\tilde \mu$ and $\tilde R$ satisfy the defining axioms of YI 
	and IY
	algebras if and only if they satisfy $\delta^2_{\rm YBH}(\phi\oplus \psi)=0$, as given in the statement. Let us consider the equation  $(\tilde \mu\otimes \mathbb 1)(\mathbb 1\otimes \tilde R)(\tilde R\otimes \mathbb 1) = \tilde R(\mathbb 1\otimes \tilde \mu)$
in $\tilde V=V \otimes {\mathbb k}[[\hbar]]/(\hbar^2)$. 
	We have for Equation~\eqref{eqn:IY}
	\begin{eqnarray*}
	(\tilde \mu\otimes \mathbb 1)(\mathbb 1\otimes \tilde R)(\tilde R\otimes \mathbb 1) &=& 
	(\mu\otimes \mathbb 1)(\mathbb 1\otimes R)(R\otimes \mathbb 1) +
	\hbar (\psi \otimes \mathbb 1)(\mathbb 1\otimes R)(R\otimes \mathbb 1) + \\
	&& \hbar (\mu \otimes \mathbb 1)(\mathbb 1\otimes \phi)(R\otimes \mathbb 1) +
	\hbar (\mu \otimes \mathbb 1)(\mathbb 1\otimes R)(\phi \otimes \mathbb 1) 
	\end{eqnarray*}
	and also
	\begin{eqnarray*}
	\tilde R(\mathbb 1\otimes \tilde \mu) &=&  R(\mathbb 1\otimes \mu) +
	\hbar \phi (\mathbb 1\otimes \mu) + \hbar R(\mathbb 1\otimes \psi) .
	\end{eqnarray*}
 Using the fact that $(V, \mu, R)$ is an IY algebra, 
  we find that $(\tilde \mu\otimes \mathbb 1)(\mathbb 1\otimes \tilde R)(\tilde R\otimes \mathbb 1) = \tilde R(\mathbb 1\otimes \tilde \mu)$ holds true if and only if 
$\delta^2_{\rm IY} (\phi\oplus\psi)=0$ 
is satisfied. One can then proceed analogously to show that the equation $(\mathbb 1\otimes \tilde \mu)(\tilde R\otimes \mathbb 1)(\mathbb 1\otimes \tilde R) = \tilde R (\tilde \mu \otimes \mathbb 1)$ if and only if 
$\delta^2_{\rm YI} (\phi\oplus\psi)=0$
 holds, up to higher order terms in $\hbar$. 
	\end{proof}

	\begin{remark}
		{\rm 
		Results analogous to the one given in Lemma~\ref{lem:IYIdeform} hold for YI and IY algebras, where just one of the mixed differentials $\delta^2_{\rm YI}$ and $\delta^2_{\rm IY} $ map to zero, respectively. 
	}
	\end{remark}

\begin{theorem}\label{thm:classification}
	Let $(V,\mu,R)$ be a braided algebra. Then the Yang-Baxter Hochschild second cohomology group classifies the infinitesimal deformations of $(V,\mu,R)$. 
\end{theorem}
\begin{proof}
	To prove the theorem, we need to show two facts. First, that Yang-Baxter Hochschild $2$-cocycles define infinitesimal deformations, and that cobounded $2$-cocycles produce equivalent 
	deformations. Second, that each infinitesimal deformation arises from a $2$-cocycle, and that if two deformations are isomorphic, then the corresponding $2$-cocycles are cobounded in Yang-Baxter Hochschild  cohomology.
The first part, that a Yang-Baxter Hochschild  $2$-cocycle 
	gives a Yang-Baxter Hochschild infinitesimal deformation,
	follows from Lemma~\ref{lem:IYIdeform}.
	
	Suppose that $(\phi,\psi)$ is cobounded in the Yang-Baxter Hochschild cohomology, where $\phi: V\otimes V\rightarrow V\otimes V$ and $\psi: V\otimes V\rightarrow V$. 
    We show that the induced infinitesimal deformation is equivalent to the original.
	By definition of first differential, the assumption  means that 
	$\phi = \delta_{\rm YB} f$ and $\psi = \delta_{\rm H} f$ for some $f:V\rightarrow V$. Then, $\tilde R$ is the trivial YB deformation and $\tilde \mu$ is the trivial algebra deformation. Then, we can construct the map $\tilde f = \mathbb 1 + \hbar f$. This is invertible in $V \otimes \mathbb k[[\hbar]]/(\hbar^2)$ with inverse $\tilde f^{-1} = \mathbb 1 - \hbar f$.  Also, we now show that both $\tilde f$ and $\tilde f^{-1}$ are braided algebra homomorphisms between the undeformed braided algebra 
	$(V, \mu, R )$
	 and the deformed braided algebra 
	 $(\tilde V, \tilde \mu, \tilde R )$.
	We show that 	$(\tilde f\otimes \tilde f) R = \tilde R(\tilde f\otimes \tilde f)$.
	For the left hand side of the previous equation we have, up to terms of higher order in $\hbar$,
	$$
	(\tilde f\otimes \tilde f) R 
	 =	R  + \hbar (f\otimes \mathbb 1)R +  \hbar (\mathbb 1\otimes f)R,
	$$
	while for the right hand side we obtain
	$$
	\tilde R(\tilde f\otimes \tilde f) = R + \hbar \phi+ \hbar R(f\otimes \mathbb1) + \hbar R(\mathbb 1\otimes f).
	$$
	The two terms are seen to be equal since $\delta^1_{\rm YB} f = \phi$. A similar inspection of the homomorphism condition for associative algebras shows also that $\tilde f \tilde \mu = \mu (\tilde f\otimes \tilde f)$. Analogously, one shows that the inverse of $\tilde f$ is a braided algebra homomorphism, completing the first part of the proof.   
	
	The statement that  if $\tilde R$ and $\tilde \mu$ produce an infinitesimal deformation of braided algebras, then $(\phi, \psi)$ is a Yang-Baxter Hochschild $2$-cocycle,  was shown in Lemma~\ref{lem:IYIdeform}.

	To complete the proof, we only need to show that if $(\phi_1, \psi_1)$ and $(\phi_2, \psi_2)$ give rise to equivalent deformations, then  they are cohomologous. Let $\tilde f$ be a map that gives the isomorphism between the braided algebra structures associated to $(\phi_1, \psi_1)$ and $(\phi_2, \psi_2)$. Then, since $\tilde f$ is an isomorphism that fixes the degree zero braided algebra structure, it follows that $\tilde f = \mathbb 1 + \hbar f$ for some $f$. Now, up to degrees higher than $1$ in $\hbar$ we have that $\tilde f^{-1} = {\mathbb 1} - \hbar f$. Since $\tilde f$ is a homomorphism of braided algebra structures it follows that 
	\begin{eqnarray}\label{eqn:cob_YB}
	\tilde R_1 &=&  (\tilde f^{-1}\otimes \tilde f^{-1}) \tilde R_2(\tilde f\otimes \tilde f),
	\end{eqnarray}
	essentially following the same steps as in the first part of the proof backward, where $\tilde R_1$ and $\tilde R_2$ are the YB operators deformed by $\psi_1$ and $\psi_2$, respectively. Moreover, we have
	\begin{eqnarray}\label{eqn:cob_Hoch}
	\tilde \mu_1 &=& \tilde f^{-1}\tilde \mu_2(\tilde f\otimes \tilde f),
	\end{eqnarray}
	where $\tilde \mu_1$ and $\tilde \mu_2$ are the associative multiplications induced by $\phi_1$ and $\phi_2$, respectively. From Equation~\eqref{eqn:cob_YB}, considering only terms up to degree $1$ in $\hbar$ we find that $\delta^1_{\rm YB} f = \psi_1 - \psi_2$. From Equation~\eqref{eqn:cob_Hoch}, considering terms up to degree $1$ we find that $\delta^1_{\rm H}f = \phi_1-\phi_2$. This completes the proof.
\end{proof}

\section{Braided multiplication and deformation cohomology}\label{sec:braidedmulti}

In this  and the next section, we discuss nontriviality of the second Yang-Baxter Hochschild cohomology groups. The third YBH cohomology groups are defined in Section~\ref{sec:3coh}
and independent from these two sections. The reader who would like to proceed to the third cohomology groups can 
go directly  to Section~\ref{sec:3coh}.

Let $(V, \mu, R)$ be a braided algebra.
In \cite{Baez} it is pointed out that the {\it braided multiplication}
$\mu R$ provides an associative multiplication, hence $(V, \mu R)$ is an associative algebra.
The diagrammatic proof is reproduced in Figure~\ref{XYassoc}. 
Furthermore, it follows that $(V, \mu R, R)$ is a braided algebra.
In this section we provide a monomorphism in second cohomology between them.
For short we denote $(V, \mu, R)$ by simply $V$ and refer $(V, \mu R, R)$ to $V_R$. Moreover, for a given braided algebra $V$ we denote by $V_{R^n}$ the braided algebra where multiplication is defined as $\mu_{R^n} := \mu \circ R^n$ and the YB operator is the same as for $V$, i.e. $R$.

\begin{figure}[htb]
\begin{center}
\includegraphics[width=3.5in]{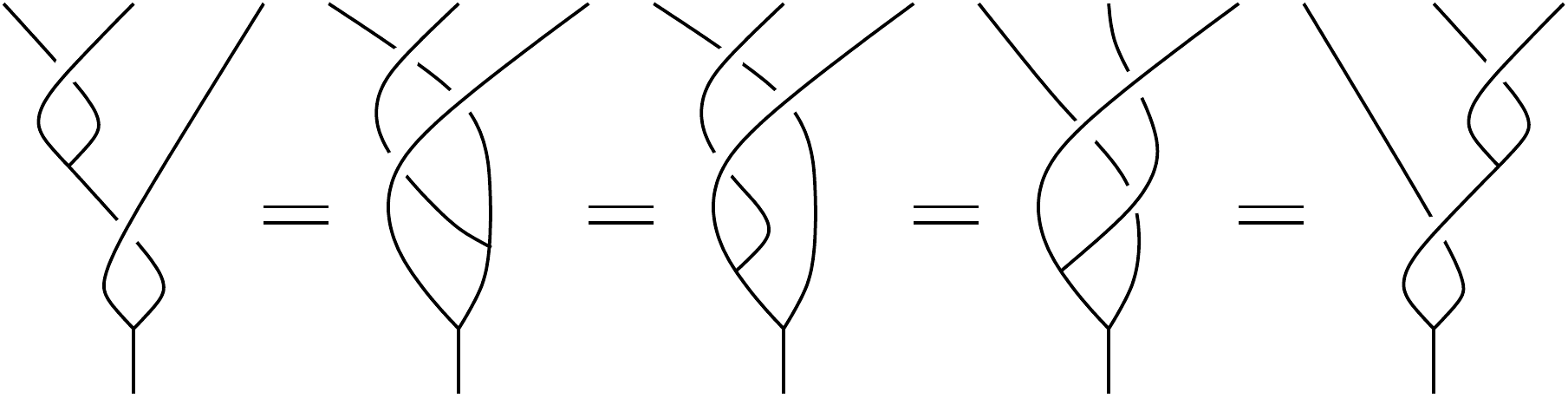}
\end{center}
\caption{}
\label{XYassoc}
\end{figure}

\begin{theorem}
There is a monomorphism
$\iota_R: H^2_{\rm YBH}(V, V) \rightarrow H^2_{\rm YBH}(V_R, V_R)$.
\end{theorem}

\begin{proof}
For $( \phi, \psi) \in C^2_{\rm YBH}(V,V)$ 
of $(V, \mu, R)$, 
define 2-cochains $(\phi_R, \psi_R)$ of $(V, \mu R, R)$ by 
$\psi_R 
:= \mu \phi + \psi R$ and 
$\phi_R 
:=\phi$, $(\psi_R, \phi_R)\in C^2_{\rm YBH}(V_R, V_R)$, as well as  
$f_R \in C^1_{\rm YBH}(V_R,V_R)$ by $f_R:=f$ for 
$f \in C^1_{\rm YBH}(V,V)$.
The map $\iota_R$  is defined by these assignments. 

First we show that $(\phi_R , \psi_R) \in Z^2_{\rm YBH}(V_R, V_R)$. 
Each of associativity, Equation~\eqref{eqn:YI}, and Equation~\eqref{eqn:IY} for $V_R$ follow from those of $V$, as shown in Figure~\eqref{XYassoc}. 
Analogously, in the infinitesimal deformation $\tilde V$  of $V$, $(\mu + \hbar \psi, R + \hbar \phi)$ satisfies these relations, from the 2-cocycle relations of $(\phi , \psi)$.
In order to show that $(\phi_R , \psi_R)$ satisfies the 2-cocycle conditions, it is sufficient to show that the infinitesimal deformation $\tilde V_R$ of $V_R$,
$(\mu_R + \hbar \psi_R, R + \hbar \phi_R)$ satisfies the associativity condition, Equation~\eqref{eqn:YI} and Equation~\eqref{eqn:IY}.
We note that 
$$\mu_R + \hbar \psi_R =\mu R + \hbar (  \mu \phi + \psi R)  \equiv (\mu + \hbar \psi) (R + \hbar \phi)
\quad {\rm modulo } \quad (\hbar^2), $$ 
and $R + \hbar \phi_R=R + \hbar \phi $.
Therefore the fact that the infinitesimal deformation $(\mu_R + \hbar \psi_R, R + \hbar\psi_R)$ giving a braided algebra 
modulo $(\hbar^2)$ follows from the assumption that $(\mu + \hbar \psi, R + \hbar \phi)$
giving the infinitesimal deformation that is a braided algebra.
Hence the claim follows.

Alternatively, it can be directly verified that 2-cocycle conditions for  $(\phi , \psi)$
implies those for  $(\phi_R , \psi_R)$ diagrammatically. For example, such computations for 
the Hochschild (associative) 2-cocycle condition are depicted in Figure~\ref{XYHoch}.
In the figure, the equalities $(A)$, $(B)$, $(C)$, $(D)$ are, respectively,
the YI, Hochschild, Yang-Baxter, and IY 2-cocycle conditions. 

To show that $\iota_R$ descends to cohomology, for 
$(\phi , \psi) =( \delta^1_{\rm YB}(f), \delta^1_{\rm H}(f))$, 
we  compute
\begin{eqnarray*}
\psi_R
&=&
\mu \delta^1_{\rm YB} (f) + \delta^1_{\rm H} (f) R \\ 
&=& [ \mu R ( f \otimes {\mathbb 1} ) + \mu R ( {\mathbb 1} \otimes f ) - 
\mu (  f \otimes {\mathbb 1} )  R - \mu  ( {\mathbb 1} \otimes f ) R ] \\
& & + [  \mu (  f \otimes {\mathbb 1} )  R + \mu  ( {\mathbb 1} \otimes f ) R - f \mu R ] \\
&=&  \mu R ( f \otimes {\mathbb 1} ) + \mu R ( {\mathbb 1} \otimes f )  - f \mu R \\
&=& \delta^1_{\rm YBH}( f ) 
\end{eqnarray*}
as desired. It is also clear that $\iota_R$ is a $\mathbb k$-module homomorphism.
To show injectivity, assume that $(\phi_R, \psi_R)$ is a coboundary, that is, 
$\psi_R=\delta^1_{\rm H}(f)$ and $\phi_R=\delta^1_{\rm YB}(f)$.
Since $\phi_R=\phi$, we have $\phi=\delta^1_{\rm YB}(f)$.
From $\psi_R=\delta^1_{\rm H}(f)$ we obtain
$\mu\phi + \psi R = \mu R ( f \otimes {\mathbb 1} ) + \mu R ( {\mathbb 1} \otimes f ) - 
 f \mu R$.
 Substituting $\phi=\delta^1_{\rm YB}(f)$ in the left hand side, we obtain 
$$  [ \mu R ( f \otimes {\mathbb 1} ) + \mu R ( {\mathbb 1} \otimes f ) - 
\mu (  f \otimes {\mathbb 1} )  R - \mu  ( {\mathbb 1} \otimes f ) R ] + \psi R =
 \mu R ( f \otimes {\mathbb 1} ) + \mu R ( {\mathbb 1} \otimes f ) - 
 f \mu R , $$ 
 which gives 
 $ \psi R = \mu ( f \otimes {\mathbb 1} ) R + \mu  ( {\mathbb 1} \otimes f ) R -  f \mu R$,
 and canceling the invertible $R$, we obtain $\psi=\delta^1_{\rm H}(f)$, thus
 $(\phi, \psi)$ is a coboundary.
\end{proof}

Inductively, we have the following.

\begin{corollary}\label{cor:VRn}
If $H^2_{\rm YBH} (V,V)\neq 0$, 
then for all $n >0$, we have $H^2_{\rm YBH} (V_{R^n},V_{R^n})\neq 0$
\end{corollary} 

\begin{remark}
	{\rm 
			Corollary~\ref{cor:VRn} shows that whenever we find a braided algebra $V$ with nontrivial YBH cohomology (as we do in Section~\ref{sec:Hopf_cohomology}), there are infinitely many braided algebras $V_{R^n}$ whose second YBH cohomology group is nontrivial as well. 
	}
\end{remark}

\begin{figure}[htb]
\begin{center}
\includegraphics[width=3.5in]{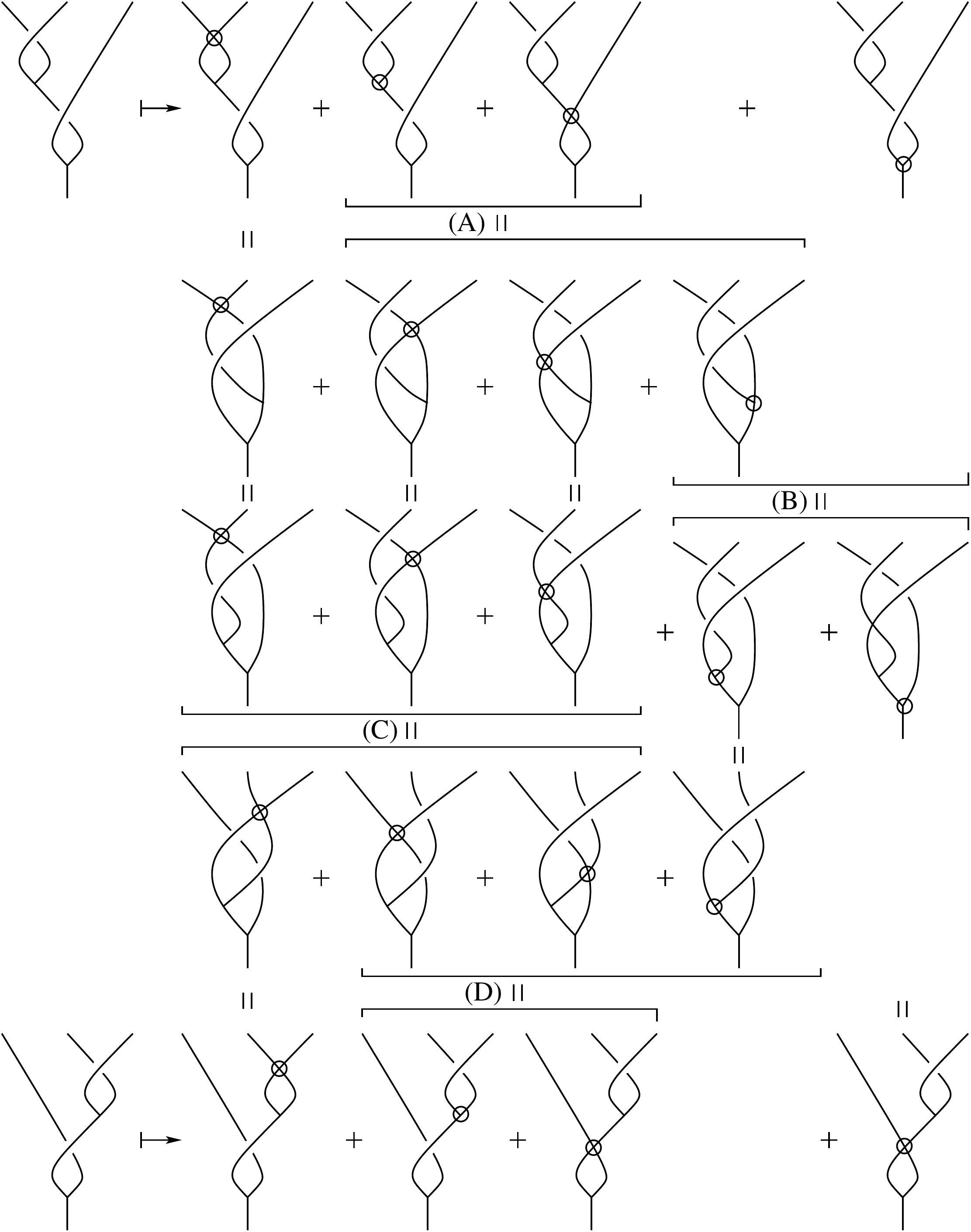}
\end{center}
\caption{}
\label{XYHoch}
\end{figure}

\section{Hopf algebra cohomology and braided algebra cohomology}\label{sec:Hopf_cohomology}

In this section we construct a homomorphism between the second cohomology group of Hopf algebras and the second YBH cohomology of their corresponding braided algebras.  Recall (\cite{ChariPressley}, Chapter 6) that the Hopf algebra second cohomology group with coefficients in $H$ characterizes Hopf algebra deformations. Roughly speaking, a $2$-cocycle for a Hopf algebra consists of a pair of cochains $(\xi, \zeta )$ such that $\xi$ deforms the algebra structure, $\zeta$ deforms the coalgebra structure, and such deformations are compatible. 

We adopt the symbols
$(\phi, \psi)$ for YBH 2-cocycles, and  
use
$(\xi, \zeta)$ for Hopf algebra 2-cocycles,
where 
$\xi \in {\rm Hom}(V^{\otimes 2}, V)$ is a deformation 2-cocycle of a Hopf algebra multiplication $\mu$, 
and $ \zeta \in {\rm Hom}(V, V^{\otimes 2})$ is a deformation 2-cocycle of a comultiplication $\Delta$.
Both multiplications for braided algebras and Hopf algebras share the same symbol $\mu$,
as it causes little confusion.

More specifically, Hopf algebra 2-cochain groups with coefficient group $H$
are defined as 
$$C^2_{\rm Hf}(H,H) = C^{2,1}_{\rm Hf}(H,H) \oplus C^{2,2}_{\rm Hf}(H,H),$$ 
where 
$C^{2,1}_{\rm Hf}(H,H)={\rm Hom}(H^{\otimes 2},H)$
and $C^{2,2}_{\rm Hf}(H,H)={\rm Hom}(H,H^{\otimes 2})$.
If $\mu$ and $\Delta$ denote the multiplication and comultiplication of a Hopf algebra $H$,
the 2-cocycles $\xi \in Z^{2,1}_{\rm Hf}(H,H)$ and $\zeta \in Z^{2,2}_{\rm Hf}(H,H)$ deform
$\mu$ and $\Delta$, respectively, so that 
$\tilde \mu=\mu + \hbar \xi$ and $\tilde \Delta = \Delta + \hbar \zeta$ defines 
a Hopf algebra structure on $H \otimes {\mathbb k}[[\hbar]]/(\hbar^2)$. 
This condition requires that $(\xi, \zeta)$ satisfies the compatibility condition 
needed to guarantee that $\tilde \Delta \tilde \mu = (\tilde \mu \otimes \tilde \mu)(\mathbb 1\otimes \tau \otimes \mathbb 1)(\tilde \Delta \otimes \tilde \Delta)$,
where $\tau$ indicates transposition of tensor factors as before.

Let $H$ be a Hopf algebra and let $(\xi, \zeta)$ 
be a $2$-cocycle with coefficients in $H$. The $2$-cocycle condition for $\xi$ with coefficients in $H$ takes the form (on simple tensors)
\begin{eqnarray*}\label{eqn:Hopf_2cocy}
\xi(x\otimes y)z + \xi(xy\otimes z) = x\xi(y\otimes z) + \xi(x\otimes yz).
\end{eqnarray*}

The morphism $\zeta : H\rightarrow H\otimes H$ is a  coalgebra cocycle (see \cite{Doi}). That is, for every $x \in H$, $\zeta$ satisfies the equation
\begin{eqnarray*}\label{eqn:Hopf_2cocy_coalgebra}
x^{(1)}\otimes \zeta(x^{(2)}) + (\mathbb 1\otimes \Delta)\zeta(x) = \zeta(x^{(1)})\otimes x^{(2)} + (\Delta\otimes \mathbb 1)\zeta(x).
\end{eqnarray*}
Additionally, from imposing that the infinitesimal deformations $\xi$ and $\zeta$ are compatible -- i.e. they define a bialgebra structure, one gets a compatibility condition involving both maps, namely
\begin{eqnarray*}\label{eqn:bialgebra}
\Delta(\xi(x\otimes y)) + \zeta(xy) = (\mu\otimes \xi + \xi\otimes \mu)\Delta^{13}(x) \Delta^{24}(y) + \zeta(x)\Delta(y) + \Delta(x)\zeta(y),
\end{eqnarray*}
where we have used the definition $\Delta^{ij}(x)$ to indicate the $4$ components tensor with $\Delta(x)$ at entries $i$ and $j$, and $1$ elsewhere. 

	Suppose that $(\xi,\zeta)$ is a Hopf algebra $2$-cocycle. It is known (see \cite{ChariPressley}) that any bialgebra deformation of $H$ is equivalent to a deformation where unit and counit are not deformed. 
	Thus, we will keep $\eta$ and $\epsilon$ to indicate unit and counit, respectively, without any 
	 deformation. 
	This 
	assumption is equivalent to that $(\xi,\zeta)$ is a normalized cocycle,
	  i.e. 
	a cocycle such that 
	$$\xi (1\otimes x) = \xi(x\otimes 1) = 0
	\quad {\rm and} \quad 
	(\mathbb 1\otimes \epsilon)\zeta(x) = (\epsilon \otimes \mathbb 1)\zeta(x) = 0$$
	for all $x\in H$. Denote by $\hat{Z}_{\rm Hf}^2 (H,H)$ the group of normalized Hopf algebra 2-cocycles.

	Below we observe that for any normalized 2-cocycle $(\xi,\zeta)$, there is a unique deformation of the antipode 
 	$\tilde S = S + \hbar S'$ that gives rise to an infinitesimal deformation of a Hopf algebra.
 	The degree 1 term for the antipode condition is as follows:

	\begin{eqnarray}\label{eqn:hexagon}
			\xi(\mathbb 1\otimes S)\Delta + \mu(\mathbb 1\otimes S)\zeta + \mu (\mathbb 1 \otimes S') \Delta =  	\xi( S \otimes \mathbb 1)\Delta + \mu(S \otimes \mathbb 1)\zeta + \mu (S' \otimes \mathbb 1) \Delta =  0.
	\end{eqnarray}
	We set $S'(x) = -S(x^{(1)})\xi(x^{(2)}\otimes S(x^{(3)}) )
	- S(x^{(1)})\mu((\mathbb 1\otimes S)\zeta(x^{(2)}) ) $. 
	Using this definition in Equation~(\ref{eqn:hexagon}) we obtain
	\begin{eqnarray*}
			\lefteqn{\xi(\mathbb 1\otimes S)\Delta(x) + \mu(\mathbb 1\otimes S)\zeta(x) + \mu (\mathbb 1 \otimes S') \Delta(x)}\\
			&=& \xi(x^{(1)}\otimes S(x^{(2)})) +  \mu(\mathbb 1\otimes S)\zeta(x) - x^{(1)}S(x^{(2)})\xi(x^{(3)}\otimes S(x^{(4)}) ) \\
			&& \hspace{55mm}  - \mu(x^{(1)}\otimes S(x^{(2)})\mu((\mathbb 1\otimes S) \zeta(x^{(3)}) ) \\
			&=&  \xi(x^{(1)}\otimes S(x^{(2)})) + \mu(\mathbb 1\otimes S)\zeta(x) - \epsilon(x^{(1)})\xi(x^{(2)} \otimes S(x^{(3)})) - \epsilon(x^{(1)}) \mu((\mathbb 1\otimes S)\zeta(x^{(2)}) \\
			&=& 0. 
	\end{eqnarray*}
	Applying the $2$-cocycle condition for $\xi$ on (the sum in Sweedler notation of terms) $S(x^{(1)}), x^{(2)}, S(x^{(3)}) $, and using the fact that $\xi$ is normalized, we find that 
	$$\xi(S(x^{(1)})\otimes x^{(2)})S(x^{(3)}) = S(x^{(1)})\xi(x^{(2)} \otimes S(x^{(3)}) ) .$$ 
	Similarly, using the $2$-cocycle condition and normalization for $\zeta$ we also find that 
	$$S(x^{(1)})\mu((\mathbb 1\otimes S)\zeta(x^{(2)}) ) = \mu((S\otimes \mathbb 1)\zeta(x^{(1)}) ) S(x^{(2)}).$$ 
	Therefore, we can rewrite $S'$ using these symmetries, and proceed in a way analogous to the previous result to obtain that $\xi( S \otimes \mathbb 1)\Delta + \mu(S \otimes \mathbb 1)\zeta + \mu (S' \otimes \mathbb 1) \Delta =  0$ as well, which completely shows that $S'$ satisfies Equation~\eqref{eqn:hexagon}. The fact that $S'$ is unique, follows from uniqueness of antipodes in Hopf algebras. 
	
	We will denote the normalized second cohomology group of $H$ with coefficients in $H$ with the symbol $\hat H^2_{\rm Hf}(H,H)  := \hat{Z}_{\rm Hf}^2 (H,H) /  ( B^1_{\rm Hf}(H,H) \cap \hat{Z}_{\rm Hf}^2 (H,H) )$, where $B^1_{\rm Hf}(H,H)$ denotes the coboundaries in Hopf cohomology. We recall that a coboundary is a pair $(\delta^{1,1}_{\rm Hf}(f),\delta^{1,2}_{\rm Hf}(f))$ where
	\begin{eqnarray*}
			\delta^{1,1}_{\rm Hf}(f) &:=& \mu(\mathbb 1\otimes f) + \mu(f\otimes \mathbb 1) - f\mu\\
			\delta^{2,1}_{\rm Hf}(f) &:=& (f\otimes \mathbb 1+\mathbb 1\otimes f)\Delta - \Delta f, 
	\end{eqnarray*}
	for some $f:H \longrightarrow H$. 

	For ease of notation, we will denote $\zeta(x)$ in Sweedler notation as for $\Delta$, the only difference being that we will use lower scripts for $\zeta$, while upper scripts for $\Delta$. Therefore, we will have $\zeta(x) = x_{(1)}\otimes x_{(2)}$. Since $(\Delta\otimes \mathbb 1)\zeta$ and $(\zeta\otimes \mathbb 1)\Delta$ are in general not the same, we introduce further a notation to distinguish the consecutive application of $\Delta$ and $\zeta$. We set 
	$$(\mathbb 1\otimes \Delta)\zeta(x) = (\mathbb 1\otimes \Delta)(x_{(1)}\otimes x_{(2)}) = x_{(1)}\otimes (x_{(2)})^{(1)}\otimes (x_{(2)})^{(2)} =: \prescript{}{(1)}x\otimes \prescript{}{(2)}x \otimes \prescript{}{(3)}x, $$
	 and similarly we set 
 	$(\zeta \otimes \mathbb 1 )\Delta(x) =: \prescript{(1)}{}x\otimes \prescript{(2)}{}x\otimes \prescript{(3)}{}x$.
	This notation allows us to forget the brackets that determine whether the internal index is lower or upper -- i.e. whether we have applied $\Delta$ or $\zeta$ first.

Let $H$ denote a Hopf algebra and let $R_H$ be the YB operator defined through 
the adjoint map 
as in Section~\ref{sec:Hopf_example}. Let $(\xi,\zeta)$ denote a Hopf algebra $2$-cocycle as described above.
We define the cochain $\Psi_{\zeta}(\xi)  : H\otimes H \rightarrow H\otimes H$ as 
\begin{eqnarray*}
	\Psi_\zeta(\xi)(x\otimes y) &=& y^{(1)}\otimes S(y^{(2)})\xi(x\otimes y^{(3)}) + y^{(1)}\otimes \xi(S(y^{(2)})\otimes xy^{(3)})\\
	&&+\prescript{}{(1)}{y} \otimes S(\prescript{}{(2)}{y})x\prescript{}{(3)}{y} + \prescript{(1)}{}{y} \otimes S(\prescript{(2)}{}{y})x\prescript{(3)}{}{y} +  y^{(1)} \otimes S'(y^{(2)})xy^{(3)}. 
\end{eqnarray*}
Observe that since $S'$ is uniquely associated to the pair $(\xi,\zeta)$, we do not indicate explicitly the dependence of $\Psi$ on $S'$, 
but rather only express its dependence on $\xi$ and $\zeta$.

%

\begin{theorem}\label{thm:Hopf_mono}
	The map $\Psi$ defined above induces a well defined homomorphism $\Psi: \hat H^2_{\rm Hf}(H,H) \rightarrow H^2_{\rm YBH}(H,H)$, $(\xi,\zeta) \mapsto (\Psi_\zeta (\xi)), \xi)$. Moreover, if $[(\xi,\zeta)] \in H^2_{\rm Hf}(H,H)$ is such that $0 \neq [\xi] \in H^2_{\rm H}(H,H)$, then $\Psi(\xi,\zeta)$ is the nonzero class in $H^2_{\rm YBH}(H,H)$.
\end{theorem}
\begin{proof}
We note  that a $2$-cochain $\phi: H\otimes H\rightarrow H\otimes H$
 is a YB $2$-cocycle for $R_H$ if and only if it is an infinitesimal deformation, i.e. if and only if $\tilde R = R_H + \hbar \phi$
 is a YB operator for $H \otimes \mathbb k[[\hbar]]/(\hbar^2)$. 
Let us consider the deformed structures $\tilde \mu = \mu + \hbar \xi$ and $\tilde \Delta = \Delta + \hbar \zeta$. They define a Hopf algebra structure (the infinitesimal deformation of $H$) since $(\xi,\zeta)$ is a Hopf $2$-cocycle by definition. 
Let  $R_{\tilde H}$ be the YB operator defined through 
the adjoint map 
as in Section~\ref{sec:Hopf_example} for the deformation $\tilde H$.
Since $R_{\tilde H}$ can be written as $R_{\tilde H} = R_H + \hbar \Psi_{\zeta}(\xi)$ 
we have that $\Psi_{\zeta}(\xi)$ satisfies the YB $2$-cocycle condition. Since $\xi$ is a Hochschild $2$-cocycle by definition of Hopf cohomology, to show that $(\Psi_\zeta(\xi),\xi)$ is a braided algebra $2$-cocycle, we need to show that the YI and IY $2$-cocycle conditions also hold. 

We apply the same argument as above to show that $(\Psi_{\zeta}(\xi), \xi)$ satisfy the YI and IY $2$-cocycle conditions.
By Lemma~\ref{lem:IYIdeform}, the deformations $\tilde \mu = \mu + \hbar \xi$ and $\tilde R = R + \hbar \Psi_\zeta (\xi)$ define a braided algebra if and only if 
$(\Psi_\zeta (\xi), \xi)$ is a YI and IY 2-cocycle.
By Lemma~\ref{lem:BA}, if $\tilde \mu$ and $\tilde \Delta$ define a (infinitesimal) Hopf algebra, 
then $\tilde\mu$ and the adjoint YB operator $\tilde R$ define a braided algebra.
Hence $(\Psi_\zeta (\xi), \xi)$ is a YI and IY 2-cocycle.

We have therefore shown that $(\Psi_\zeta(\xi),\xi) \in Z^2_{\rm YBH}(H,H)$. We now need to show that this map descends to cohomology. To do this, we need to prove that if $(\xi,\zeta) = (\delta^{1,1}_{\rm Hf}(f), \delta^{1,2}_{\rm Hf}(f))$ is cobounded and normalized, then $(\Psi_\zeta(\xi), \xi)$ is cobounded in YBH cohomology as well. Using the definition of Hochschild cohomology and YBH cohomology, we have that it is enough to show that $\Psi_\zeta(\xi) = \delta^2_{\rm YB}(f)$ for some $1$-cochain $f: H\rightarrow H$. 

We compute first $\prescript{}{(1)}x\otimes \prescript{}{(2)}x \otimes \prescript{}{(3)}x$ and $\prescript{(1)}{}x\otimes \prescript{(2)}{}x\otimes \prescript{(3)}{}x$ for such a cobounded $\zeta$. We have
\begin{eqnarray*}
	\prescript{}{(1)}x\otimes \prescript{}{(2)}x \otimes \prescript{}{(3)}x &=& (\mathbb 1\otimes \Delta)\delta_{\rm Hf}^{1,2}(f)(x)\\
	&=& -f(x)^{(1)}\otimes f(x)^{(2)} \otimes f(x)^{(3)} + f(x^{(1)}) \otimes x^{(2)} \otimes x^{(3)}\\
	&& + x^{(1)} \otimes f(x^{(2)})^{(1)} \otimes f(x^{(2)})^{(2)}
\end{eqnarray*}
and 
\begin{eqnarray*}
	\prescript{(1)}{}x\otimes \prescript{(2)}{}x\otimes \prescript{(3)}{}x &=& (\mathbb 1\otimes \delta_{\rm Hf}^{1,2}(f))\Delta x\\
	&=& -x^{(1)}\otimes f(x^{(2)})^{(1)}\otimes f(x^{(2)})^{(2)} + x^{(1)}\otimes f(x^{(2)})\otimes x^{(3)}\\
	&& + x^{(1)} \otimes x^{(2)} \otimes f(x^{(3)}).
\end{eqnarray*}
Before proving that $\Psi_\zeta(\xi)$ is a YB coboundary, we also show that for $(\xi,\zeta) = (\delta^{1,1}_{\rm Hf}(f), \delta^{1,2}_{\rm Hf}(f))$, we have $Sf -fS + S' = 0$, where $S'$ is the degree 1 term of the unique antipode deformation associated with $(\xi,\zeta)$. In fact, we have
\begin{eqnarray*}
		S'(x) &=& -S(x^{(1)})\xi(x^{(2)}\otimes S(x^{(3)}) ) - S(x^{(1)})\mu((\mathbb 1\otimes S)\zeta(x^{(2)}) ) \\
		&=&f(S(x)) - S(f(x)) + \epsilon(f(x^{(2)})) S(x^{(1)}) \\
		&=& f(S(x)) - S(f(x)), 
\end{eqnarray*}
where we have used the fact that $\epsilon(f(x^{(2)})) S(x^{(1)}) = 0$ 
because 
$\delta^1_{\rm Hf}(f)$ is normalized.

Now, we use these results to compute $\Psi_\zeta(\xi)$:
\begin{eqnarray*}
	\Psi_\zeta(\xi)(x\otimes y) &=& y^{(1)}\otimes S(y^{(2)})f(xy^{(2)}) - y^{(1)}\otimes S(y^{(2)})xf(y^{(3)}) - y^{(1)}\otimes S(y^{(2)})f(x)y^{(3)}\\
	&& + y^{(1)}\otimes f(S(y^{(2)})xy^{(3)}) - y^{(1)}\otimes S(y^{(2)})f(xy^{(3)}) - y^{(1)}\otimes f(S(y^{(2)}))xy^{(3)}\\
	&& - f(y)^{(1)}\otimes S(f(y)^{(2)})xf(y)^{(3)} + f(y^{(1)}) \otimes S(y^{(2)})xy^{(3)}\\
	&& + y^{(1)}\otimes S(f(y^{(2)})^{(1)})xf(y^{(2)})^{(2)} - y^{(1)}\otimes S(f(y^{(2)})^{(1)})xf(y^{(2)})^{(2)}\\
	&& + y^{(1)}\otimes S(f(y^{(2)}))xy^{(3)} + y^{(1)} \otimes S(y^{(2)})xf(y^{(3)}) + y^{(1)}\otimes S'(y^{(2)})xy^{(3)}\\
	&=& -\delta_{\rm YBH}^1(f)(x\otimes y) - y^{(1)} \otimes f(S(y^{(2)}))xy^{(3)}\\
	&& + y^{(1)}\otimes S(f(y^{(2)}))xy^{(3)} +  y^{(1)}\otimes S'(y^{(2)})xy^{(3)}\\
	&=& \delta_{\rm YB}^1(-f)(x\otimes y),
\end{eqnarray*}
where in the last step we have used the fact that $Sf(x) - f(S(x)) + S'(x) = 0$. This shows that the map $\Psi$ induced by the correspondence $\Psi_\zeta(\xi)$ is well defined in cohomology.

To complete the proof of the theorem, we need to show that if $(\xi,\zeta)$ is a nontrivial deformation as an algebra, then its corresponding braided algebra deformation is nontrivial as well. This comes from the fact that if $(\Psi_\zeta(\xi), \xi)$ is cobounded in YBH cohomology, by definition it follows that $\xi$ is cobounded in Hochschild cohomology.
\end{proof}

\begin{example}\label{ex:group_algebra}
	{\rm 
	By Theorem~\ref{thm:classification},
	if a cocycle deformation of a braided algebra by 
	a pair of 2-cocycles
	$(\phi, \psi)$ defines a nontrivial deformation, then 
	$(\phi, \psi)$  is a non-coboundary.
		While it is known that over fields of characteristic zero there are no nontrivial 
		deformations for the group algebra \cite{ChariPressley}, for group algebras over more general rings there exist nontrivial deformations of the algebra structure 
		\cite{SiegelWith}. Therefore, the corresponding braided  algebras are deformed nontrivially.
		Hence  in such cases the  second cohomology group of YBH cohomology  is nontrivial. 
	}
\end{example}

\begin{example}\label{ex:Hopf_algebra}
	{\rm 
	Let us now consider a more general setting than that of Example~\ref{ex:group_algebra} above. For an algebraic group $G$, let $\mathcal F(G)$ denote the algebra of regular functions. Then it is known (see \cite{ChariPressley}) that 
	there are nontrivial deformations of $\mathcal F(G)$ as  a Hopf algebra,
	though the coalgebra structure has trivial deformation.
	Applying Theorem~\ref{thm:Hopf_mono}  we obtain nontrivial deformations (and therefore nontrivial elements in the braided algebra second cohomology group according to Theorem~\ref{thm:classification}) of the associated braided  algebra. 
}
\end{example}

\section{Third cohomology and quadratic deformations} \label{sec:3coh} 

In this section we consider the problem of extending an infinitesimal deformation to higher degrees in $\hbar$. In the case of associative algebras \cite{Gerst} and Lie algebras \cite{Nij-Rich}, it is well known that the obstruction to extending an infinitesimal deformation to higher degrees lies in the third cohomology group of the algebra, therefore giving a simple criterion for the existence of such extensions. In fact, if the third cohomology group is trivial, it follows that any infinitesimal deformation can be extended. In this situation it is also said that the deformation is {\it integrable}. This holds true also when dealing with YB cohomology \cite{Eisermann}. 

First we define the third cohomology groups for braided algebras,
and then give their interpretation as obstructions to deformations.

\subsection{The third differentials}

Let $V$ be a
braided
 algebra.
 We define
 $$ C^4_{\rm YI}(V,V) =
 C^{4,4}_{\rm YBH}(V,V)  \oplus 
C^{4,3}_{\rm YI}(V,V)  \oplus 
C^{4,2}_{\rm YI}(V,V)  \oplus 
 C^{4,2}_{\rm IY}(V,V) \oplus
C^{4,1}_{\rm YBH}(V,V)  .
$$
We define $\delta^3_{\rm YI} (V,V) : C^3_{\rm YI} (V,V) \rightarrow  C^4_{\rm YI} (V,V)$ by 
	$\delta^3_{\rm YI} = \delta^3_{\rm YB} \oplus \delta^{3,1}_{\rm YI}  \oplus  \delta^{3,2}_{\rm YI}
	\oplus \delta^{3,3}_{\rm YI}   \oplus \delta^3_{\rm H}$, where each direct summand of the differential is defined as follows.
For $\alpha \in  C^{3,2}_{\rm YI}(V,V)$ and $\beta \in C^{3,3}_{\rm YBH}(V,V)$, 
define 
$\delta^{3,1}_{\rm YI} : C^{3,3}_{\rm YBH}(V,V)  \oplus C^{3,2}_{\rm YI}(V,V)
\rightarrow C^{4,3}_{\rm YI}(V,V)   $
by
\begin{eqnarray*}
\delta^{3,1}_{\rm YI}(\beta \oplus \alpha  ) 
&=& 
( {\mathbb 1} \otimes R ) ( \alpha \otimes {\mathbb 1} ) ( {\mathbb 1}^{\otimes 2}  \otimes R )
+( {\mathbb 1} \otimes \alpha ) ( R \otimes {\mathbb 1}^{\otimes 2} ) (  {\mathbb 1} \otimes R \otimes {\mathbb 1} ) ( {\mathbb 1}^{\otimes 2}  \otimes  R ) 
\\ & & 
+
( {\mathbb 1}^{\otimes 2}  \otimes   \mu ) (    {\mathbb 1} \otimes R \otimes {\mathbb 1} ) 
( R \otimes {\mathbb 1}^{\otimes 2} ) ( {\mathbb 1} \otimes \beta) 
+ ( {\mathbb 1}^{\otimes 2}  \otimes   \mu ) ( \beta \otimes  {\mathbb 1} )
( {\mathbb 1}^{\otimes 2}  \otimes   R ) (   {\mathbb 1} \otimes R \otimes {\mathbb 1} ) \\
& & - 
\beta ( \mu \otimes  {\mathbb 1}^{\otimes 2} ) 
-
  ({\mathbb 1} \otimes R )( R \otimes     {\mathbb 1} ) (\alpha \otimes {\mathbb 1} ) - (R\otimes \mathbb 1)(\mathbb 1\otimes \alpha)(R\otimes \mathbb 1^{\otimes 2})(\mathbb 1\otimes R\otimes \mathbb 1) .
 \end{eqnarray*}
Define 
 $C^{3,2}_{\rm YI}(V,V) \oplus C^{3,2}_{\rm IY}(V,V)
\rightarrow C^{4,2}_{\rm IY}(V,V)
$
for $\alpha \in C^{3,2}_{\rm YI}(V,V)$ and $\alpha ' \in C^{3,2}_{\rm IY}(V,V)$ 
by
 \begin{eqnarray*}
\delta^{3,2}_{\rm YI}( \alpha \oplus \alpha ') 
&=& 
\alpha ' ( \mu \otimes     {\mathbb 1}^{\otimes 2} ) 
+ 
( \mu \otimes  {\mathbb 1} )  (   {\mathbb 1} \otimes R ) (\alpha  \otimes     {\mathbb 1} )
+ 
(  {\mathbb 1} \otimes \alpha ) ( \mu \otimes     {\mathbb 1} )
(   R \otimes {\mathbb 1}^{\otimes 2} ) 
(   {\mathbb 1} \otimes R \otimes {\mathbb 1} ) \\
& & 
- \alpha (   {\mathbb 1}^{\otimes 2} \otimes     \mu ) 
- 
(  {\mathbb 1} \otimes \mu ) ( R \otimes  {\mathbb 1} ) (  {\mathbb 1} \otimes \alpha '   ) 
-
( \alpha '  \otimes {\mathbb 1} ) (  {\mathbb 1} \otimes \mu )
(  {\mathbb 1}^{\otimes 2}  \otimes R )
(    {\mathbb 1} \otimes R \otimes {\mathbb 1} ) . 
 \end{eqnarray*}
Define further
 $\delta^{3,3}_{\rm YI} : C^{3,2}_{\rm YBH}(V,V)  \oplus C^{3,1}_{\rm YBH}(V,V)
\rightarrow C^{4,2}_{\rm YBH}(V,V)   $
for $\alpha \in C^{3,2}_{\rm YBH}(V,V) $ and $\gamma \in  C^{3,1}_{\rm YBH}(V,V)$
by
\begin{eqnarray*}
\delta^{3,3}_{\rm YI}(\alpha \oplus \gamma) 
&=& R (\gamma \otimes {\mathbb 1}) 
+\alpha  ( {\mathbb 1} \otimes \mu \otimes  {\mathbb 1}  )  + ( {\mathbb 1} \otimes  \mu )(R \otimes  {\mathbb 1}) ( {\mathbb 1} \otimes \alpha ) \\
& & 
-\alpha (\mu \otimes  {\mathbb 1}^{\otimes 2} ) 
- ({\mathbb 1} \otimes \mu)  (\alpha \otimes {\mathbb 1})  ( {\mathbb 1}^{\otimes 2} \otimes R)
- ( {\mathbb 1} \otimes \gamma) (R \otimes  {\mathbb 1}^{\otimes 2} ) (  {\mathbb 1} \otimes R \otimes {\mathbb 1}) ( {\mathbb 1}^{\otimes 2} \otimes R) 
.
\end{eqnarray*}

The differentials $\delta^{3,i}_{\rm YI}$, $i=1,2,3$, have  diagrammatic representations depicted in 
Figures~\ref{YII},  \ref{YYI}, and \ref{IYY}, respectively.
In Figure~\ref{YII}, from the top initial diagram labeled (I) to the bottom terminal figure labeled 
(T) there are two sequences of equalities that use Equation~\eqref{eqn:YI} . 
The two sequences are labeled for left hand side and right hand side by $(L1)$, $(L2)$, $(L3)$ and $(R1)$, $(R2)$, 
respectively.
Between (I) and (L1), a trivalent vertex goes through an over-arc at a crossing, introducing two 
crossings in (L1). At the moment when the trivalent vertex and a crossing merge, 
a 5-valent vertex appears as depicted in (LD1) marked by a small circle.
This circled 5-valent vertex with top 3 edges and bottom 2 segments represents 
$\alpha \in  C^{3,2}_{\rm YBH}(V,V)={\rm Hom}_{\mathbb k}(V^{\otimes 3},V^{\otimes 2})$.
When the diagram of (LD1) is read from top to bottom, there is a crossing at the top with two parallel lines to its left. This represents the first factor $( {\mathbb 1} \otimes R )$
in the first term of 
$\delta^{3,1}_{\rm YI}(\beta \oplus \alpha) $.
The second factor $( \alpha \otimes {\mathbb 1} ) $ represents the 5-valent vertex with one string to the right,
and the last factor $(  {\mathbb 1}^{\otimes 2}  \otimes R )$
represents the crossing at the bottom of (LD1). 
The four positive terms in $\delta^{3,1}_{\rm YI}(\beta \oplus \alpha) $
are represented by (LD$n$) for $n=1,2,3,4$ in this order, and the negative terms are represented by 
the right hand side diagrams (RD$n$) for $n=1,2$.

\begin{figure}[htb]
\begin{center}
\includegraphics[width=3in]{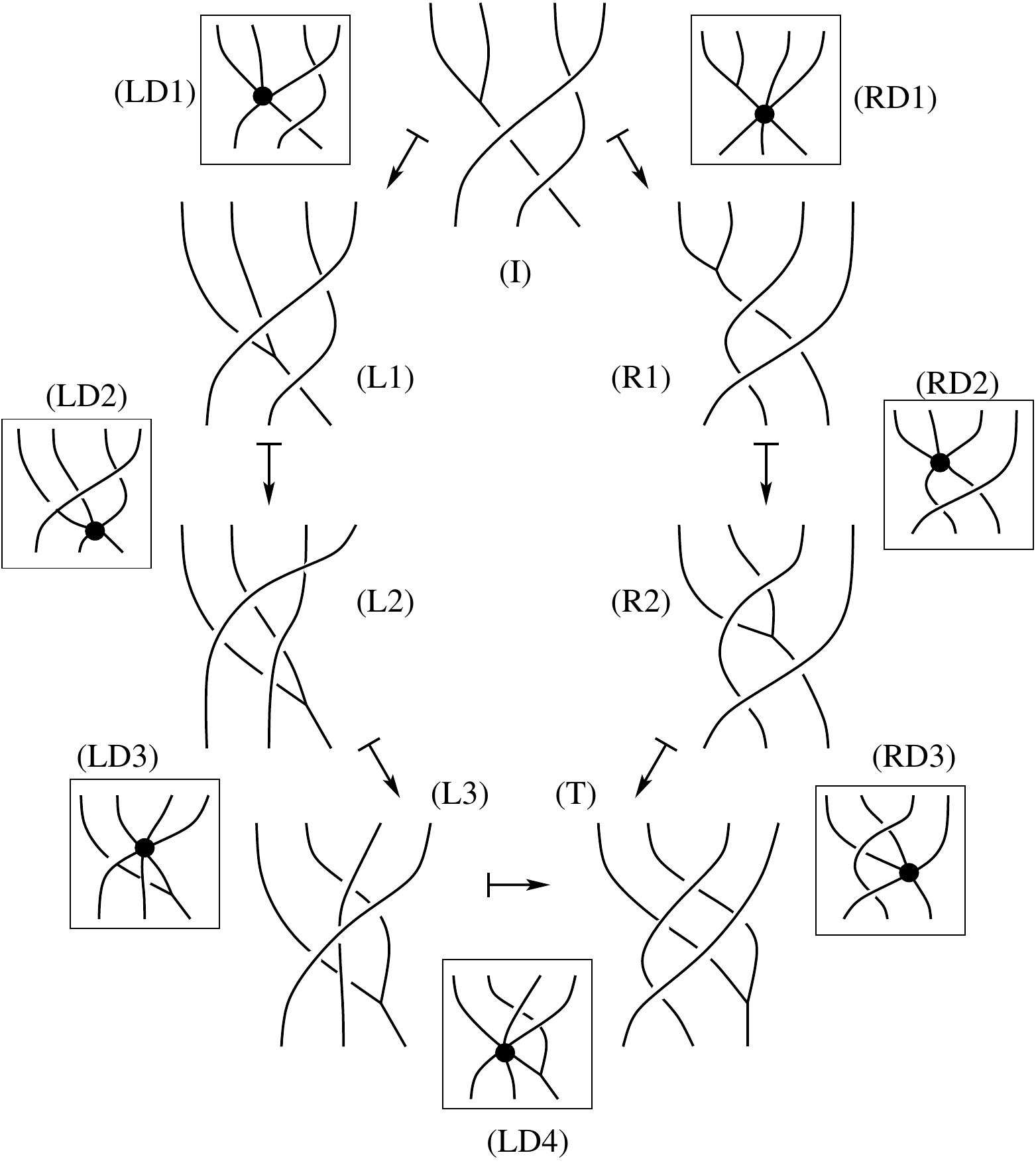}
\end{center}
\caption{}
\label{YII}
\end{figure}

\begin{figure}[htb]
\begin{center}
\includegraphics[width=3in]{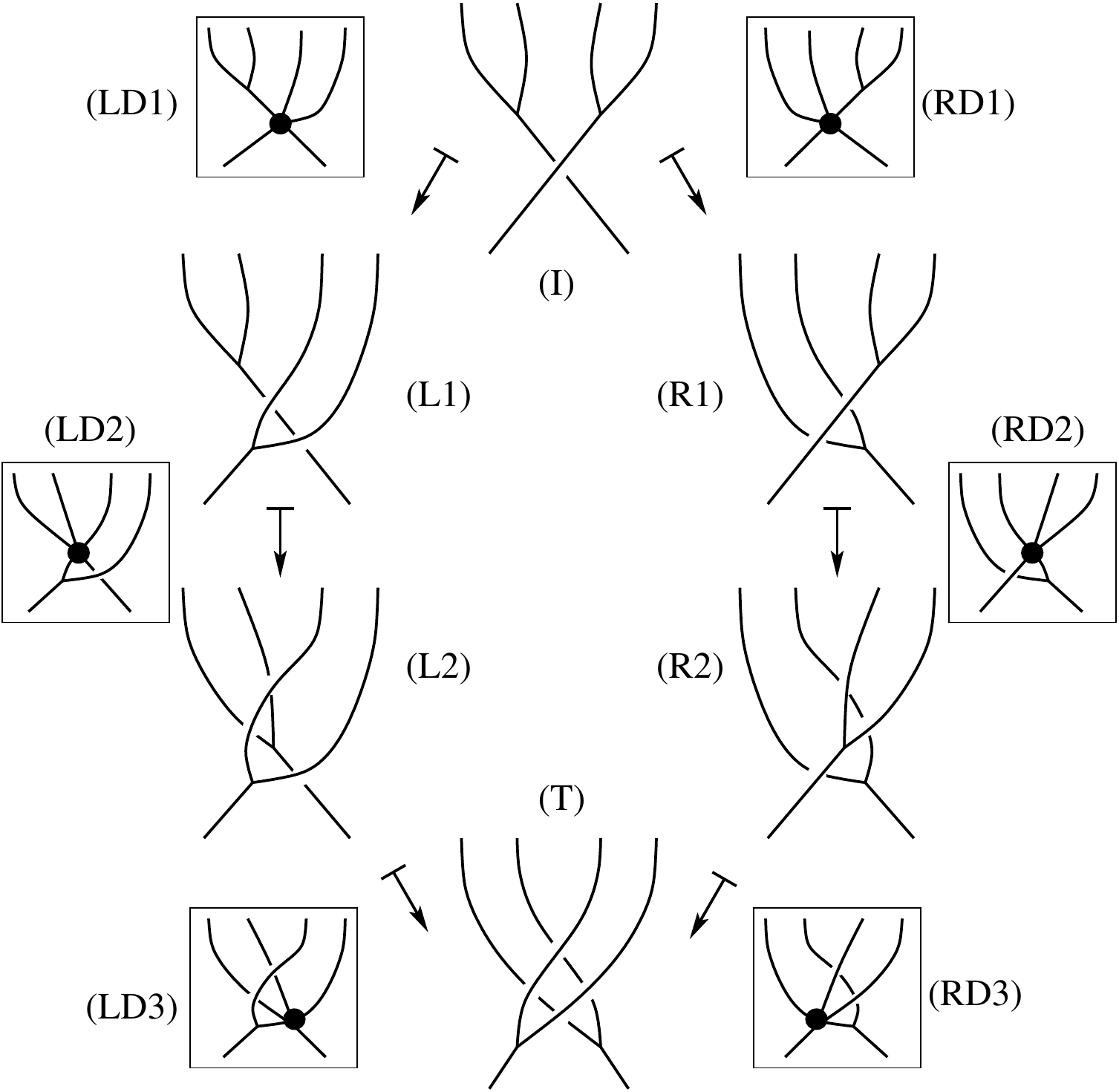}
\end{center}
\caption{}
\label{YYI}
\end{figure}

\begin{figure}[htb]
\begin{center}
\includegraphics[width=3in]{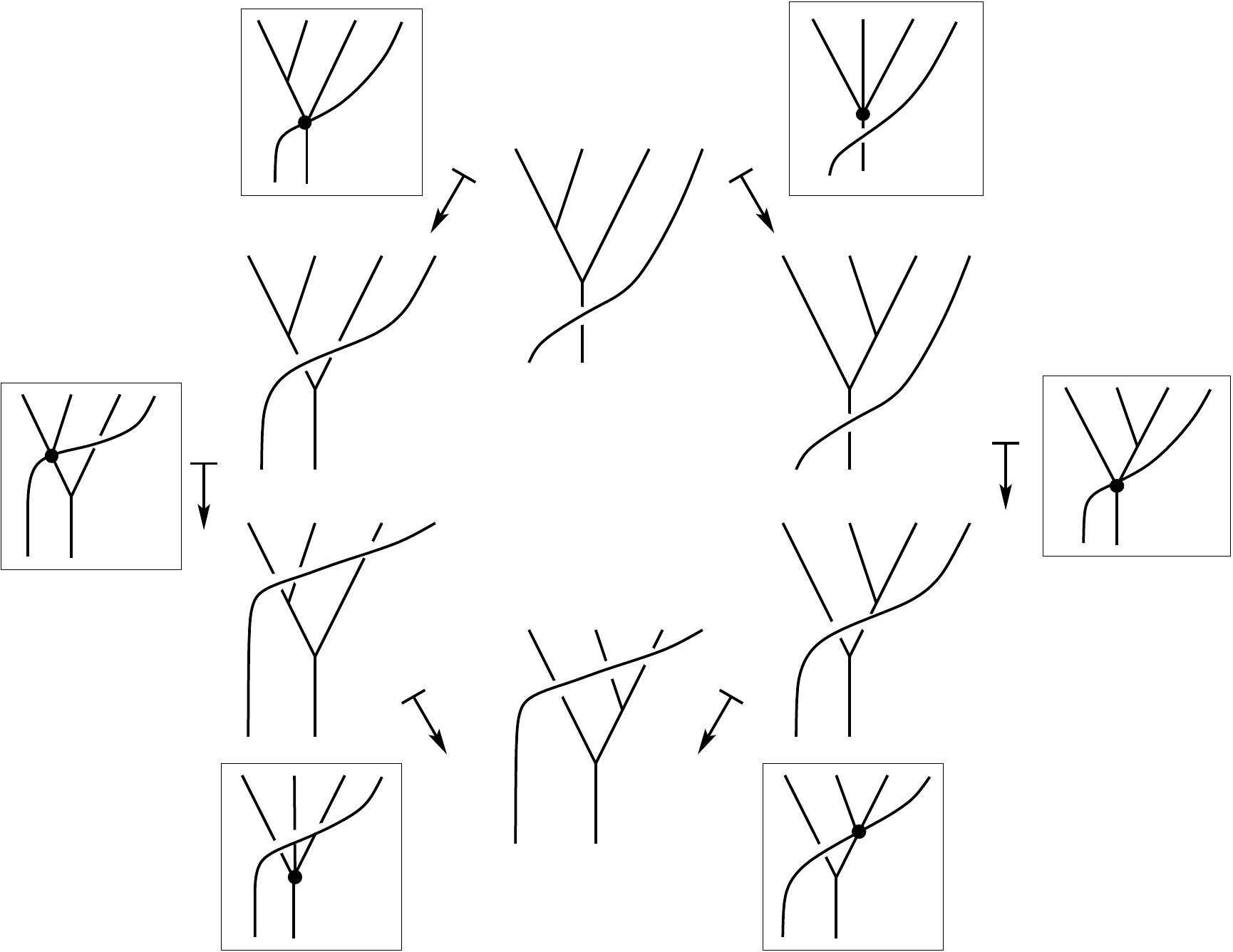}
\end{center}
\caption{}
\label{IYY}
\end{figure}

\begin{proposition}
The sequence
$$  C_{\rm YBH}^2(V,V) 
\stackrel{\delta^2_{\rm YBH}}{\longrightarrow}
C_{\rm YBH}^3(V,V) 
\stackrel{\delta^3_{\rm YI}}{\longrightarrow}
 C_{\rm YI}^4(V,V) 
 $$
 forms a cochain complex.
\end{proposition}

\begin{proof}
We show that 
$\delta^{3,2}_{\rm YI}  (\alpha \oplus \alpha' ) = 0$,
$\alpha = \delta^2_{\rm YI} (\psi \oplus \phi)$ and  $\alpha'=\delta^2_{\rm IY} (\psi \oplus \phi)$. 
We use the diagrammatics in Figure~\ref{YYIdiff2}. 
The 3-cochain 
$\alpha ' $
 is represented by a circled 4-valent (3 in, 1 out) vertex in the diagram (1). 
The assumption
$\alpha ' = \delta^2_{\rm IY }(\psi \oplus \phi )$ 
 is represented by the right side of the figure with labels
(1-1), (1-2), (2-1), (2-2) and (2-3), where the linear combination of the terms 
(1-1) $ + $ (1-2) $ - $ (2-1) $ -$  (2-2) $ -$ (2-3) represents 
 the assumed 2-cocycle condition. 
Similarly (2) represents $\alpha$ and 
(2-4) $+$ (2-5) $-$ (3-1) $-$ (3-2) $-$ (3-3) represents $\alpha = \delta^2_{\rm YI}(\psi \oplus \phi)$.
One finds the canceling pair (2-1) and (2-5). 
One finds each circled 4-valent  vertex  appearing in the same position twice, they cancel 
in pairs, indicating 
$\delta^{3,2}_{\rm IY}  (\alpha \oplus \alpha ' ) = 0 $.
The other cases are similar.
\end{proof}

Similar definitions where the IY condition is used instead of the YI condition give a differential $\delta^3_{\rm IY}$. As it will be proved below, they contain the obstructions to higher order deformations corresponding to Equation~\eqref{eqn:YI} and Equation~\eqref{eqn:IY}, respectively.

\begin{figure}[htb]
\begin{center}
\includegraphics[width=3in]{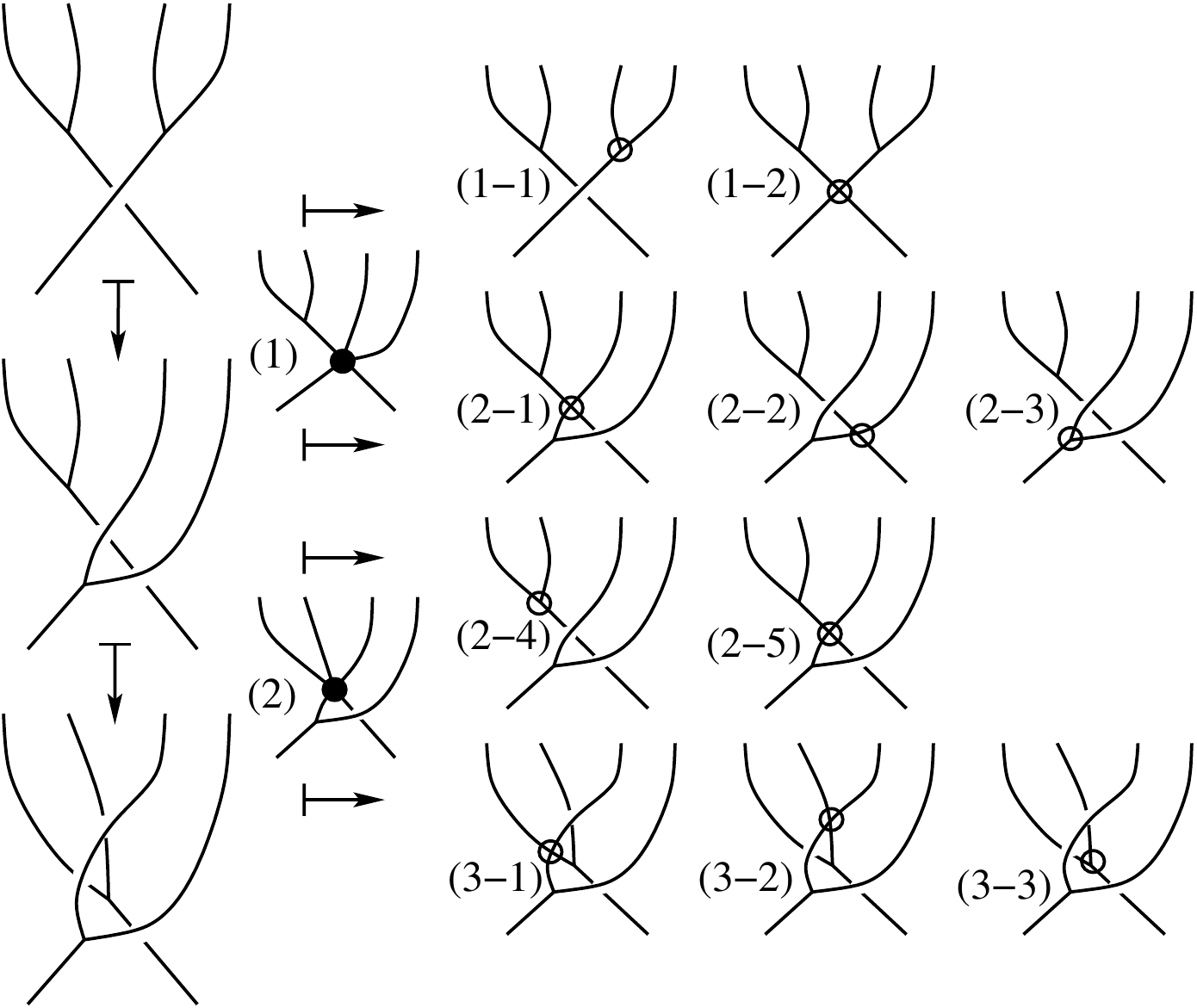}
\end{center}
\caption{}
\label{YYIdiff2}
\end{figure}

\subsection{Quadratic deformation for Hochschild cohomology}

In this section we review the second degree deformations and the third cohomology group 
of associative algebras \cite{Gerst}. 
These methods  are applied in the following section to develop an analogous theory for Yang-Baxter Hochschild cohomology. We provide a diagrammatic proof of a known result   in Appendix~\ref{AppendixA}, 
which will be used to prove  the results of Section~\ref{sec:quadratic_deform}.

	Let $(V,\mu)$ be an algebra over ${\mathbb k}$, and let $\psi=\psi_1$ denote a Hochschild $2$-cocycle,
	$\psi_1 \in Z^2_{\rm H}(V,V)$. 
As before we set  $(\tilde V, \tilde \mu)$, where $\tilde V= V \otimes {\mathbb k}[[\hbar]] /(\hbar^2)$
and $\tilde \mu = \mu + \hbar \psi_1$.
Next we consider $(\hat V, \hat \mu)$ where 
$\hat V= V \otimes {\mathbb k}[[\hbar]]/(\hbar^3)$ and 
 $\hat  \mu = \mu + \hbar \psi_1 + \hbar^2 \psi_2$.
	One computes the degree two term of $\hbar$ to be 
\begin{eqnarray*}
& & \hbar^2 [ \{  \mu (\psi_2 \otimes \mathbb{1} ) +  \psi_2 (\mu \otimes \mathbb{1} ) 
-  \mu (\mathbb{1}  \otimes \psi_2) -  \psi_2 (\mathbb{1}  \otimes \mu) ) \} 
+ \{ \psi_1 (\psi_1 \otimes \mathbb{1} ) - \psi_1 (\mathbb{1}  \otimes \psi_1) ) \} ]
  \\
& & = 
\hbar^2 [ \delta^2 (\psi_2) + \{  \psi_1 (\psi_1 \otimes \mathbb{1} ) - \psi_1 (\mathbb{1}  \otimes \psi_1) \} ] .
\end{eqnarray*}
Set $\Theta_2 (\psi_1) = \psi_1 (\psi_1 \otimes \mathbb{1} ) - \psi_1 (\mathbb{1}  \otimes \psi_1)$.
Then this computation  shows the following.

\begin{lemma} \label{lem:Hdeform2}
 The above defined $(\hat V, \hat \mu)$ is an algebra such that 
 $\hat \mu \equiv_{(\hbar^2)} \mu$
 if and only if 	$\delta^2 (\psi_2) +\Theta_2 (\psi_1) =0$. 
\end{lemma}

 In this case we say that $(\hat V, \hat \mu)$ is a 
 degree 2 deformation of $(V, \mu)$.
Next, we recall the following.

\begin{lemma}[\cite{Gerst}] \label{lem:H3cocy}
$\delta^3_{\rm H}\Theta_2 (\psi_1) =0$. 
\end{lemma}
 
By this lemma,  $\Theta_2 (\psi_1) $ is a Hochschild 3-cocycle, representing an element 
$[\Theta_2 (\psi_1) ] \in H^3_{\rm H}(V,V)$.
If  $\delta^2 (\psi_2) +\Theta_2 (\psi_1) =0$ as in Lemma~\ref{lem:Hdeform2},
then $\Theta_2 (\psi_1) $ is a coboundary, $[\Theta_2 (\psi_1) ] =0$, and the degree 1 deformation $(\tilde V, \tilde \mu)$
of $(V, \mu)$ further deforms to $(\hat V, \hat \mu)$.
In particular, if $H^3(V,V)=0$, then $(V, \mu)$  deforms to $(\hat V, \hat \mu)$.
Thus $H^3(V,V)$ can be regarded as obstruction to quadratic deformation.

We provide a diagrammatic proof of Lemma~\ref{lem:H3cocy} in Appendix~\ref{AppendixA}. We use similar techniques for the third Yang-Baxter Hochschild cohomology group in the following section.

	Lastly, we hereby recall the obstruction to higher degree deformations for associative algebras, as it motivates some of the results found below (see Theorem~\ref{lem:higher}). 
		In fact, suppose $(A,\mu)$ is an algebra and $\hat \mu = \sum_{i=0}^k \hbar^i\mu_i$ is a deformation of degree $k$ (here $\mu_0 = \mu$). Consider the degree $k+1$ deformation $\tilde \mu = \hat \mu + \hbar^{k+1}\mu_{k+1}$. Then, the associativity condition for $\tilde \mu$ is automatically satisfied in degree at most $k$, while in degree $k+1$ takes the form
	\begin{eqnarray*}
		\lefteqn{\sum_{t=1}^k \mu_t(\mu_{k+1-t}\otimes \mathbb 1) + \mu_0(\mu_{k+1}\otimes \mathbb 1) + \mu_{k+1}(\mu_0\otimes \mathbb 1)}\\
		&=&  \sum_{t=1}^k \mu_t(\mathbb 1\otimes \mu_{k+1-t}) + \mu_0(\mathbb 1\otimes \mu_{k+1})  + \mu_{k+1}(\mathbb 1\otimes \mu_0), 
	\end{eqnarray*}
	which can be rewritten as 
	$$
	\delta_H^2(\mu_{k+1}) + \Theta_{k+1} =0,
	$$
	where $\Theta_{k+1}$ contains terms not involving $\mu_{k+1}$. Similarly as above, it holds that $\delta^3_H(\Theta_{k+1}) = 0$, and the obstruction to extend a degree $k$ deformation to a degree $k+1$ deformation lies in the third cohomology group \cite{Gerst}.
	
	In this case we do not provide a diagrammatic proof of these classic results, as the analogue approach for braided algebras becomes increasingly complicated to handle. However, a different approach following the calculus of deviations of Markl-Stasheff, as described in Section~\ref{sec:quadratic_deform}, can be employed to extend the analogue of this result to the case of higher degree deformations for YBH cohomology as in the case of associative algebras. While the direct approach is not feasible at higher degrees, we still believe that it is valuable to provide a direct proof of these results of YBH cohomology at least at degree $2$. We will therefore due so, before providing the general construction. 

\subsection{Higher order deformations for Yang-Baxter Hochschild homology}\label{sec:quadratic_deform}

We  investigate the obstructions to higher deformations for braided  algebras, and obtain integrability criteria at least in some special cases of interest. 
This requires additional obstructions for the compatibility among higher degree deformation terms
lying in the third YBH cohomology groups.
First, we pose the following inductive definition, where Defintion~\ref{def:inf_deform} serves as the base case $n=1$.

\begin{definition}\label{def:higher_deform}
	{\rm 
	Let $(V,\mu, R)$ be a braided algebra, and let $(\tilde V, \tilde \mu, \tilde R)$ be a deformation of degree $n$. Then, we say that $(\hat V, \hat \mu, \hat R, )$ is a deformation of $(V,\mu, R)$ of order $n+1$ if it is a braided algebra over the ring $\mathbb k[[\hbar]]/(\hbar^{n+2})$ and, in addition, $\hat R \cong_{(\hbar^{n+1})} \tilde R$ and $\hat \mu \cong_{(\hbar^{n+1})} \tilde \mu$. 
}
\end{definition}

We will use the notation $R_n = \sum_{i=0}^n \hbar^i \phi_i$ and $\mu_n = \sum_{i=0}^n\hbar^i \psi_i$, where $\phi_0 := R$ and $\psi_0 := \mu$. For all $r\geq 2$ we also set the triples of integers
$$
\Gamma_r := \{(i,j,k)\ |\ i+j+k = r,\  i, j, k \geq 0, \ i,j,k\neq r \}. 
$$
Then, we define 
\begin{eqnarray*}
\Xi^{\rm YI}_r := \sum_{(i,j,k)\in \Gamma_r} (\psi_i\otimes \mathbb 1)(\mathbb 1\otimes \phi_j)(\phi_k\otimes 1), 
& &  \Omega^{\rm YI}_r := \sum_{(p,q,0)\in \Gamma_r} \phi_p(\mathbb 1\otimes \psi_q) , \\
  \Xi^{\rm IY}_r := \sum_{(i,j,k)\in \Gamma_r} (\mathbb 1\otimes \psi_i)(\phi_j\otimes \mathbb 1)(\mathbb 1\otimes \phi_k), & &  \Omega^{\rm IY}_r := \sum_{(p,q,0)\in \Gamma_r} \phi_p(\psi_q\otimes \mathbb 1).
  \end{eqnarray*}
  Lastly, we define 
  $$\Theta_r = \sum_{(i,j,k)\in \Gamma_r} (\phi_i\otimes \mathbb 1)(\mathbb 1\otimes \phi_j)(\phi_k\otimes \mathbb 1) - (\mathbb 1\otimes \phi_i)(\phi_j\otimes \mathbb 1)(\mathbb 1 \otimes \phi_k)$$ and $\Lambda_r = \sum_{(p,q,0)\in \Gamma_r} \mu_p(\mu_q\otimes \mathbb 1) - \mu_p(\mathbb 1\otimes \mu_q)$. 

\begin{theorem}\label{lem:higher}
	Let $(V,\mu, R)$ be a braided algebra, and let $(\tilde V, \mu_n, R_n)$ be a deformation of degree $n$. Then $(V,  \mu_{n+1}, R_{n+1})$ is a deformation of $(V,\mu, R)$ of order $n+1$ if and only if 
	\begin{center}
		$
	\begin{cases}
		\delta^2_{\rm YB}\phi_{n+1} + \Theta_{n+1} = 0\\
		\delta^2_{\rm YI}(\phi_{n+1},\psi_{n+1}) + \Xi^{\rm YI}_{n+1} - \Omega^{\rm YI}_{n+1} = 0\\
		\delta^2_{\rm IY}(\phi_{n+1},\psi_{n+1}) + \Xi^{\rm IY}_{n+1} - \Omega^{\rm IY}_{n+1} = 0\\
		\delta^2_{\rm H}\psi_{n+1} + \Lambda_{n+1} =0.
	\end{cases}
	$
	\end{center}
\end{theorem}
\begin{proof}
	The proof consists of some straightforward, albeit long, computations. We show some details regarding the third equation of the system. We set $R_{n+1} = \sum_{i=0}^{n+1}\hbar^i \phi_i$ and $\mu_{n+1} = \sum_{i=0}^{n+1}\hbar^i \mu_i$, as above. Since by hypothesis $(R_n,\mu_n)$ defines a braided algebra deformation of $(V,\mu, R)$, Equation~\eqref{eqn:IY} on $R_{n+1}$ and $\mu_{n+1}$ gives an obstruction only on terms of degree $n+1$ in $\hbar$. Then, writing only the terms of degree $n+1$ in $\hbar$ we have 
	\begin{eqnarray*}
	0 &=& (\psi_{n+1}\otimes \mathbb 1)(\mathbb 1\otimes \phi_0)(\phi_0\otimes \mathbb 1)
	+ (\psi_0\otimes \mathbb 1)(\mathbb 1\otimes \phi_{n+1})(\phi_0\otimes \mathbb 1) +(\psi_0\otimes \mathbb 1)(\mathbb 1\otimes \phi_0)(\phi_{n+1}\otimes \mathbb 1)\\
	&& + \sum_{(i,j,k)\in\Gamma_{n+1}} (\psi_i \otimes \mathbb 1)(\mathbb 1\otimes \phi_j)(\phi_k\otimes \mathbb 1) 
	- \sum_{(p,q,0)\in\Gamma_{n+1}}\phi_p(\mathbb 1\otimes \psi_q)\\
	&=& \delta^2_{\rm IY}(\phi,\psi) 
	+ \Xi^{\rm IY}_{n+1} 
	- \Omega^{\rm IY}_{n+1}, 
	\end{eqnarray*}
which is the third equation. Therefore,  $(V,  \mu_{n+1}, R_{n+1})$ satisfies Equation~\ref{eqn:IY} if and only if 
	$$\delta^2_{\rm IY}(\phi_{n+1},\psi_{n+1}) + \Xi^{\rm IY}_{n+1} - \Omega^{\rm IY}_{n+1} = 0$$
	 is satisfied. The other cases are treated analogously by direct inspection. 
\end{proof}

\begin{lemma} \label{lem:3cocy}
The obstruction to extending an infinitesimal deformation of braided algebras to a quadratic deformation lies in the third YBH cohomology group. More concisely, we have 
$$\delta_{\rm YBH}^3 (\Theta_2 ) =
\delta_{\rm YBH}^3 (\Xi^{\rm YI}_2 - \Omega^{\rm YI}_2 )=
		\delta_{\rm YBH}^3 ( \Xi^{\rm IY}_2 - \Omega^{\rm IY}_2 ) =
		\delta_{\rm YBH}^3 ( \Lambda_2 )=0. $$
\end{lemma}

\begin{proof}
The proof is given in Appendix~\ref{AppendixB}.
\end{proof}

The following result follows immediately from Lemma~\ref{lem:3cocy}.

\begin{corollary}\label{cor:quadratic}
	If $H_{\rm YBH}^3(V,V) = 0$, then any infinitesimal deformation can be extended to a quadratic deformation.
\end{corollary}
\begin{proof}
		By Lemma~\ref{lem:3cocy}, 
		$\Theta_2$,
		$\Xi^{\rm YI}_2 - \Omega^{\rm YI}_2$, 
		$ \Xi^{\rm IY}_2 - \Omega^{\rm  IY}_2$,
		and $\Lambda_2$ are 3-cocycles, representing elements of 
		$H^3_{\rm YBH}(V,V)$.
		If  the equalities  in Theorem~\ref{lem:higher} hold, then these 3-cocycles are coboundaries. 
		In particular, if $H_{\rm YBH}^3(V,V)=0$, then $(V, \mu, R)$  deforms to $(\hat V, \hat \mu, \hat R)$.
		Thus, if $H^3_{\rm YBH}(V,V) = 0$ there is no obstruction to extending an infinitesimal deformation to a quadratic deformation. 
\end{proof}

Lemma~\ref{lem:3cocy} is generalized to the extension of any degree $n$ deformation to a degree $n+1$ deformation by considering Markl-Stasheff calculus of deviations. 

\begin{theorem}\label{thm:higher_YBH_deg_n}
	Let $\bar \psi = \sum_{i=0}^n \hbar^i \psi_i$ and $\bar \phi =  \sum_{i=0}^n \hbar^i \phi_i$ be a degree $n$ deformation of $\mu = \psi_0$ and $R = \phi_0$. Let $\psi_{n+1}$, $\phi_{n+1}$, $\Theta_{n+1}$, $\Xi^{\bullet}_{n+1}$, $\Omega^{\bullet}_{n+1}$, $\Lambda^{\bullet}_{n+1}$ be as in Lemma~\ref{lem:higher}. Then, we have 
	\begin{eqnarray*}
		\delta^3_{\rm YBH}(\Theta_{n+1}\oplus(\Xi^{\rm IY}_{n+1}-\Omega^{\rm IY}_{n+1})\oplus(\Xi^{\rm YI}_{n+1}-\Omega^{\rm YI}_{n+1})\oplus\Lambda_{n+1})  = 0.
	\end{eqnarray*}
	In other words, the obstruction to extending a degree $n$ braided algebra deformation to a degree $n+1$ deformation lies in the third YBH cohomology group. 
\end{theorem}

The proof of this result is based on the use of the diagrams in Figure~\ref{YII} and Figure~\ref{IYY}, following the same approach in \cite{SZ-BC}. In fact, the figures can be used to compute the deviations for higher degrees to hold, from which we can infer that the obstruction to extendind a deformation to a higher degree holds up to the third differentials. Details can be found in \cite{SZ-BC}. 

Then, Corollary~\ref{cor:quadratic} is immediately further generalized as follows. 

\begin{corollary}
			If $H^3_{\rm YBH}(V,V) = 0$, then any  infinitesimal deformation of $V$ can be extended to a full deformation of arbitrary degree. 
\end{corollary}
		
We include the proof of Lemma~\ref{lem:3cocy} in Appendix~\ref{AppendixB} below 
since
this proof is more diagrammatic and different from methods used in \cite{SZ-BC}. Moreover, the graphs used in the present proof were used
in \cite{SZ-BC} in a different manner.


\appendix

\section{Proof of Lemma~\ref{lem:H3cocy}}\label{AppendixA}

This is a well known result in the literature, and the innovative component of this proof is a  purely diagrammatic approach, which is employed extensively in the proofs in Appendix~\ref{AppendixB}.

In Figure~\ref{hoch3}, the pentagonal diagram is depicted  that represents 
the Hochschild 3-cocycle condition. The leftmost tree represents 
the parenthesis structure $((xy)z)w$ for algebra monomials $x,y,z,w$, and the vertices of the pentagon list all possibilities of parentheses structures. 
Applying associativity corresponds to each directed edge, which represents 
a 3-cocycle $\alpha$ as depicted. The arguments involved in the associativity,
$x,y,z$, are substituted in the 3-cocycle as depicted.
The figure indicates that the Hochschild 3-cocycle condition corresponds to this figure.

\begin{figure}[htb]
	\begin{center}
		\begin{overpic}[width=2.5in]{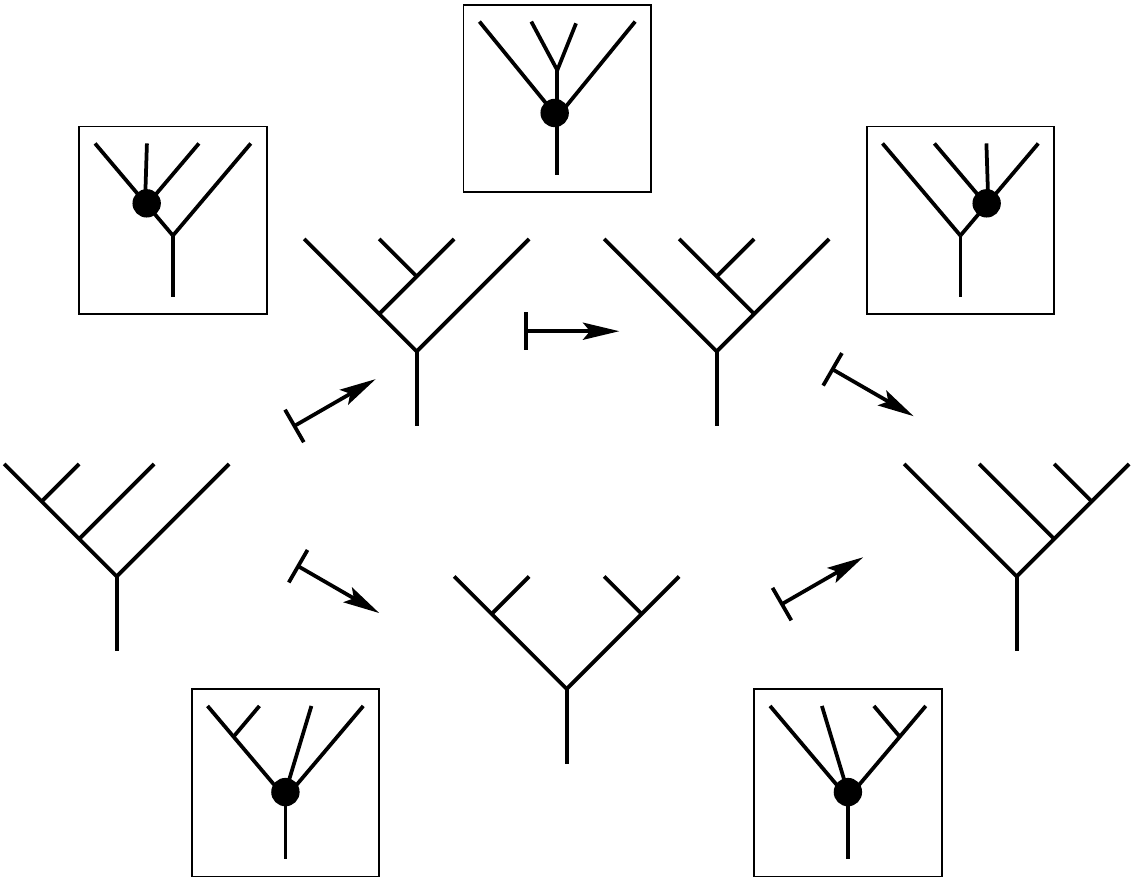}
			\put(-1,39){$x \ \, y \  \, z \  \, w $}
			\put(-5,69){$\alpha (x \otimes y \otimes z) w$}
			\put(31,80){$\alpha (x \otimes yz \otimes w) $}
			\put(66,69){$x \alpha (y \otimes z \otimes w) $}
			\put(7,-4.5){$\alpha (xy \otimes z \otimes w) $}
			\put(58,-4.5){$\alpha (x \otimes y \otimes zw) $}
		\end{overpic}
	\end{center}
	\caption{}
	\label{hoch3}
\end{figure}

\begin{figure}[htb]
\begin{center}
\includegraphics[width=3in]{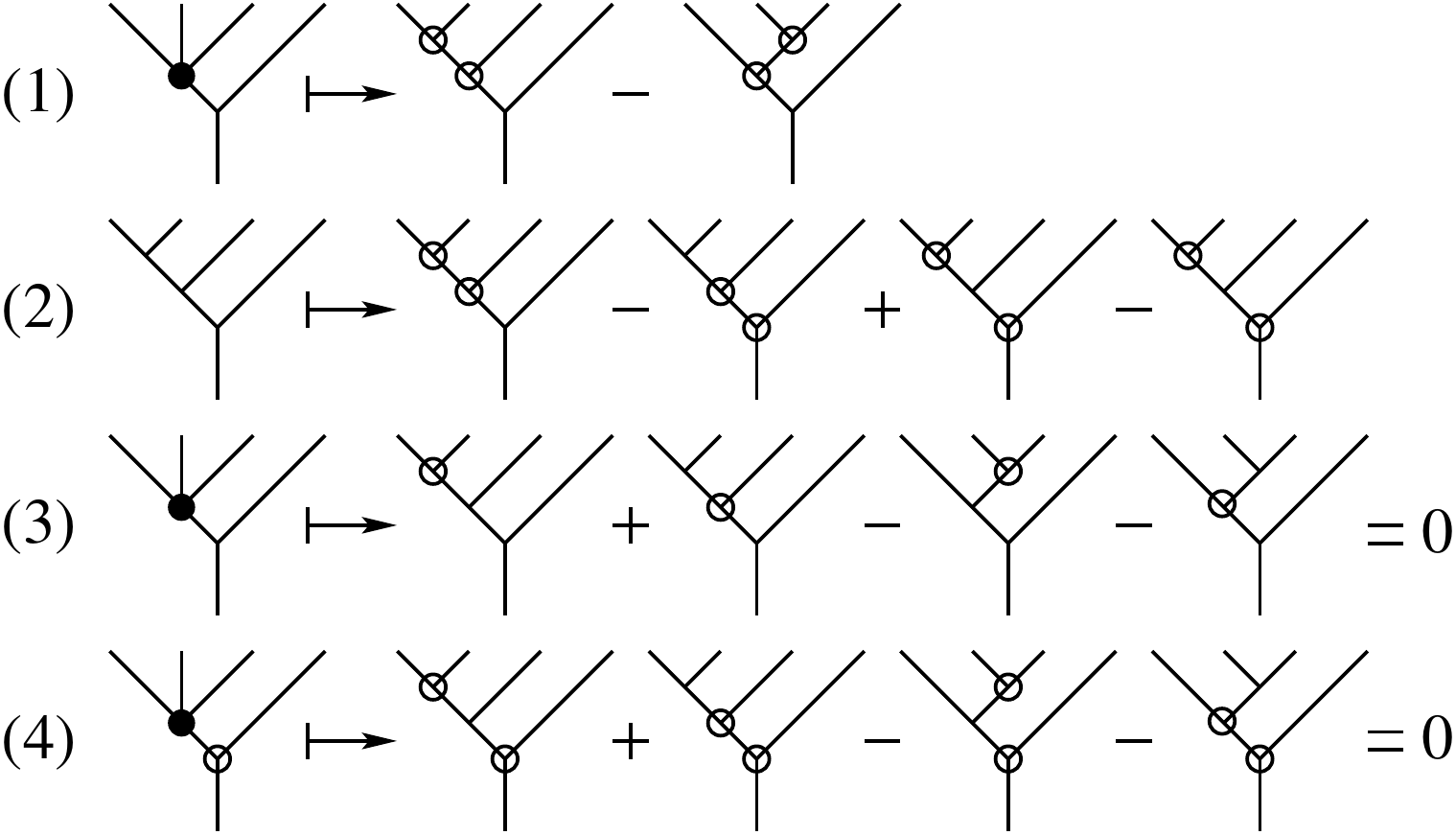}
\end{center}
\caption{}
\label{gers1}
\end{figure}

\begin{figure}[htb]
\begin{center}
\includegraphics[width=4in]{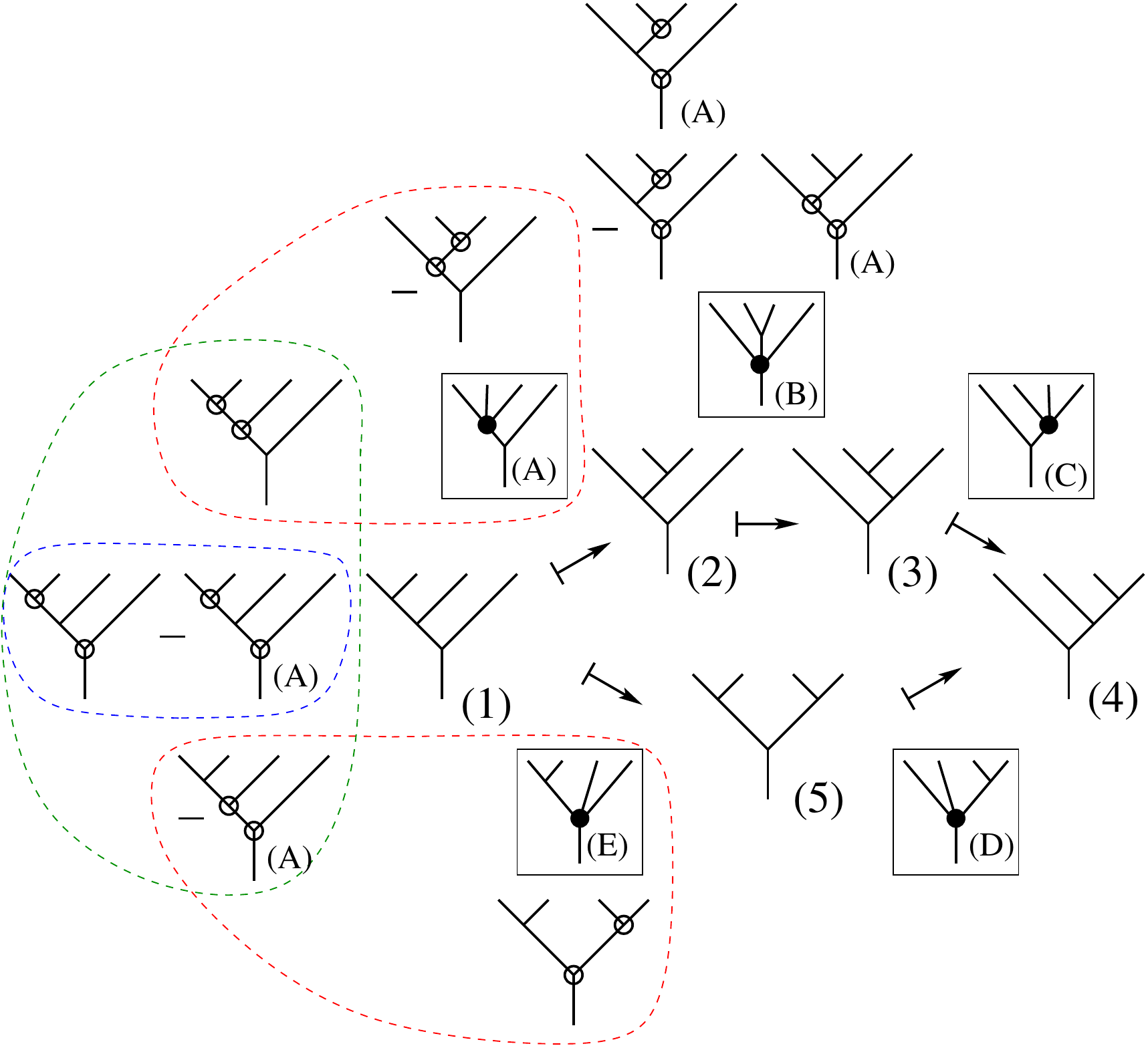}
\end{center}
\caption{}
\label{gers2}
\end{figure}

\begin{figure}[htb]
\begin{center}
\includegraphics[width=3in]{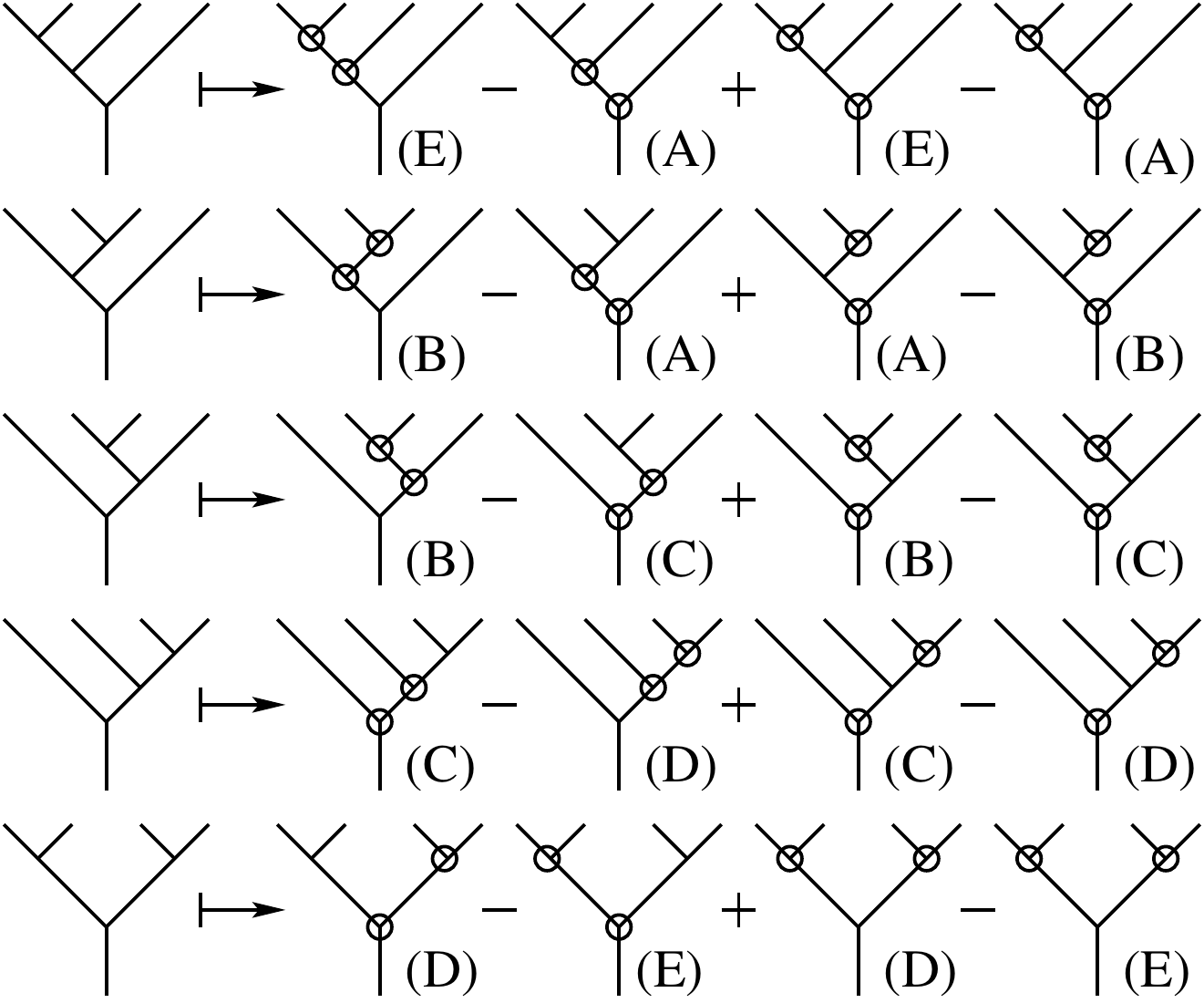}
\end{center}
\caption{}
\label{gers3}
\end{figure}

In Figure~\ref{gers1} (1), the substitution of the 
term $\psi_1(\psi_1 \otimes {\mathbb 1} ) - \psi_1( {\mathbb 1} \otimes  \psi_1)$
into a 3-cocycle, 
$\alpha (\psi_1) = \psi_1(\psi_1 \otimes {\mathbb 1} ) - \psi_1( {\mathbb 1} \otimes  \psi_1)$,
is depicted, where a white circle represents $\psi_1$. 
At each tree diagram of the pentagon, we place all possible pairs of circles to trivalent vertices,
as depicted in (2), and make a formal linear combination with $\pm$ signs.
In (2), the first two terms appear as the first term of (1), the second term appears in the $3$-cocycle labeled (E) in Figure~\ref{gers2}, and the last two canceling pairs do not appear in 3-cocycle terms. 
All terms cancel as follows.

In Figure~\ref{gers2}, the circled tree diagrams are depicted at the left most tree diagram.
The red dotted circle at the 3-cocycle diagram represents the group of two diagrams in Figure~\ref{gers1} (1). One of them is a circled diagram of the tree (1) in Figure~\ref{gers2}, and the other term corresponds to the tree (2). The blue dotted circle groups a canceling pair.
The green dotted circle represents all circled tree diagrams of the tree (1), as depicted in Figure~\ref{gers1} (2). The bottom red circle represents the substitution of 
the 3-cocycle condition depicted in (E). 

In Figure~\ref{gers3}, all circled diagrams from all tree diagrams are depicted.
The formal sum of all these terms corresponds to the differential of $\Theta_2(\psi_1)$,
and we show that this sum vanishes.
This follows from the 2-cocycle condition of $\psi_1$. 
An incident of a 2-cocycle condition is depicted in Figure~\ref{gers1} (3), for the tree diagram Figure~\ref{gers2} (1). In Figure~\ref{gers1}  (4), the 2-cocycle condition with an additional circle is depicted, which also vanishes. This vanishing equality (4) corresponds to 
Figure~\ref{gers2} (A), and its terms are depicted in Figure~\ref{gers2} with label (A).
Thus the sum of the four circled tree diagrams labeled in Figure~\ref{gers2} vanishes 
by the 2-cocycle condition. 

In Figure~\ref{gers3}, all terms on the right-hand side are labeled by 
(A) through (E). The four terms together  labeled by the same letter, then, vanish
by the 2-cocycle condition, completing the proof. 

\begin{figure}[htb]
\begin{center}
\includegraphics[width=4in]{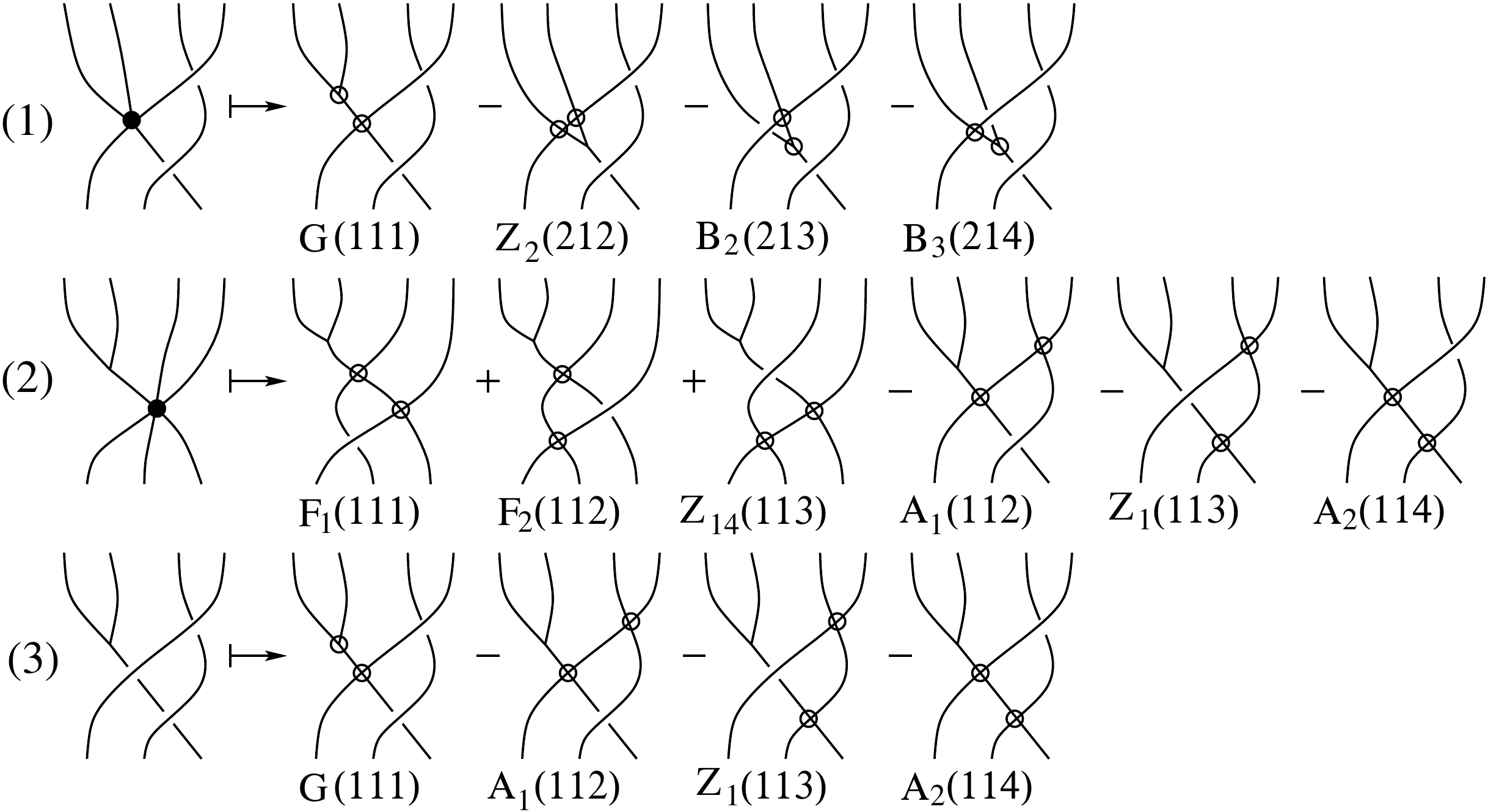}
\end{center}
\caption{}
\label{IIYdiff0}
\end{figure}

\section{Proof of Lemma~\ref{lem:3cocy}}\label{AppendixB}

The proof follows an argument similar to Lemma~\ref{lem:Hdeform2}.
The 3-cocycle $\alpha$ and $\beta$ are  represented by  black vertices of valency 2 and 3, respectively, in square-framed diagrams in Figure~\ref{YII}.
The first term $(LD1)$ of the differential on the left hand side represents 
$({\mathbb 1} \otimes R)(\alpha \otimes {\mathbb 1} ) ({\mathbb 1} ^{\otimes 2} \otimes R)$.
The 3-cocycle condition $\delta^3_{\rm YI} (\beta \oplus \alpha)  =0 $ is represented by 
the expression $(LD1) + (LD2) + (LD3) + (LD4) - (RD1) - (RD2) - (RD3)$. 

In Figure~\ref{IIYdiff0} (1), the first term $(LD1)$ in Figure~\ref{YII} for a 3-cocycle condition of YI algebra is depicted at the left hand side. The labels are used later. 
We substitute  $\Xi^{\rm YI}_{2} - \Omega^{\rm YI}_{2}$ for $\alpha$  in the term $(LD1)$, 
then we obtain maps represented by the right hand side of Figure~\ref{IIYdiff0} (1). 
Similarly, when we substitute $\Xi^{\rm YI}_{2} - \Omega^{\rm YI}_{2}$ for $\beta$ in 
$(RD1)$ in Figure~\ref{YII}, then we obtain the right hand side of Figure~\ref{IIYdiff0} (2).
When these terms in (1) and (2) 
that have two circles on the diagrams in $(I)$,
 in 
Figure~\ref{YII}, then we obtain the right hand side of Figure~\ref{IIYdiff0} (3). 
This list in the right hand side of Figure~\ref{IIYdiff0} (3) appear in Figure~\ref{IIY123} 
at the top left, above the dotted line.
Below the dotted line are canceling pairs that are used to apply 3-cocycle conditions.

All terms, represented by square-framed diagrams in Figure~\ref{YII}, 
that appear in the 3-cocycle condition, are substituted with  $\Xi^{\rm YI}_{2} - \Omega^{\rm YI}_{2}$,
and grouped together according to each diagram of Figure~\ref{YII}, from the initial diagram (I) to  the terminal diagram (T) through $(L1)$ through $(L3)$ and back to (I) through $(R2), (R1)$, 
in Figures~\ref{IIY123} through \ref{IIY67} above the dotted lines.
The terms in the right hand side $(R1)$ and $(R2)$ receive negative signs (left hand side - right hand side). 

\begin{figure}[htb]
\begin{center}
\includegraphics[width=6in]{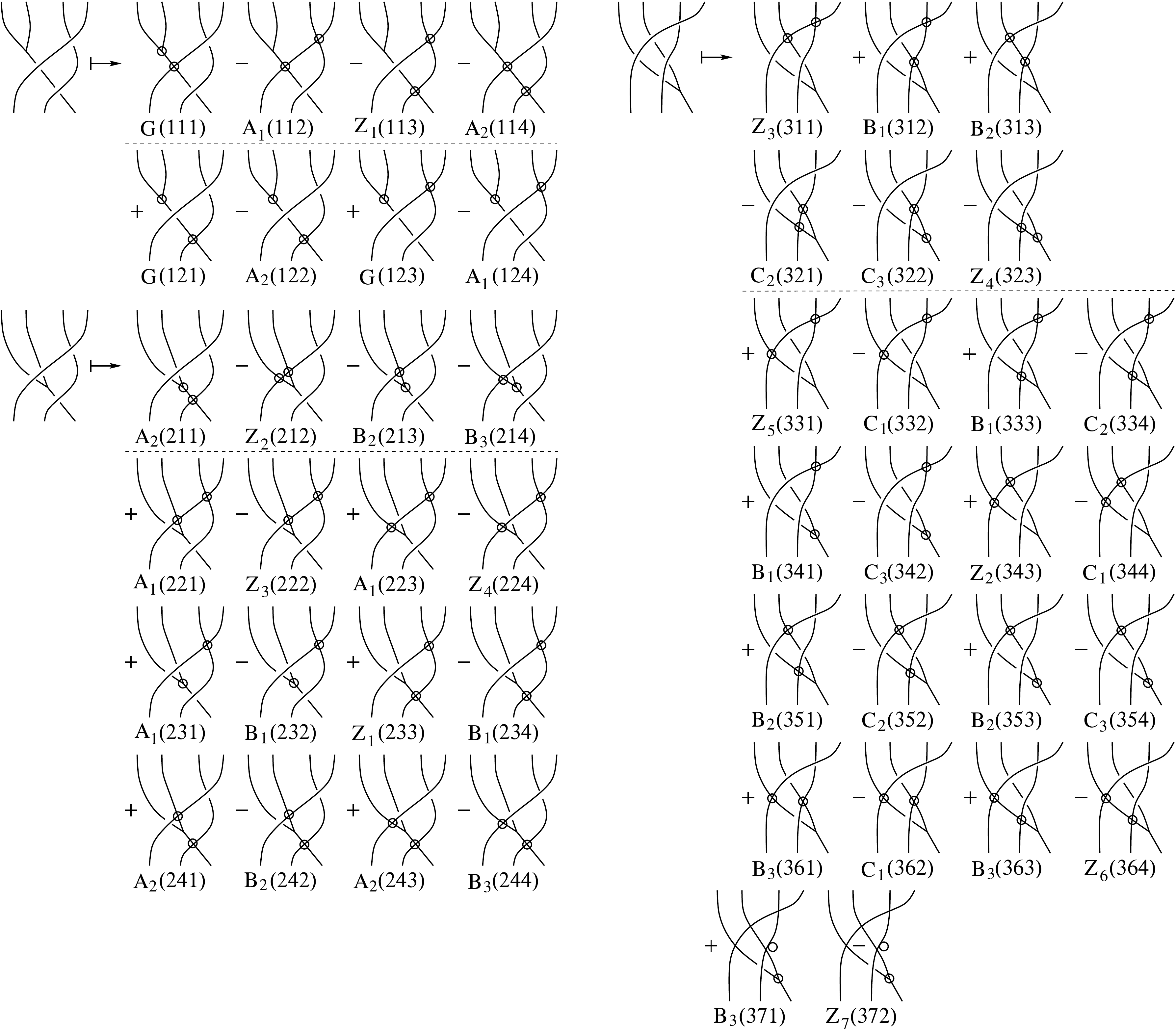}
\end{center}
\caption{}
\label{IIY123}
\end{figure}

\begin{figure}[htb]
\begin{center}
\includegraphics[width=6in]{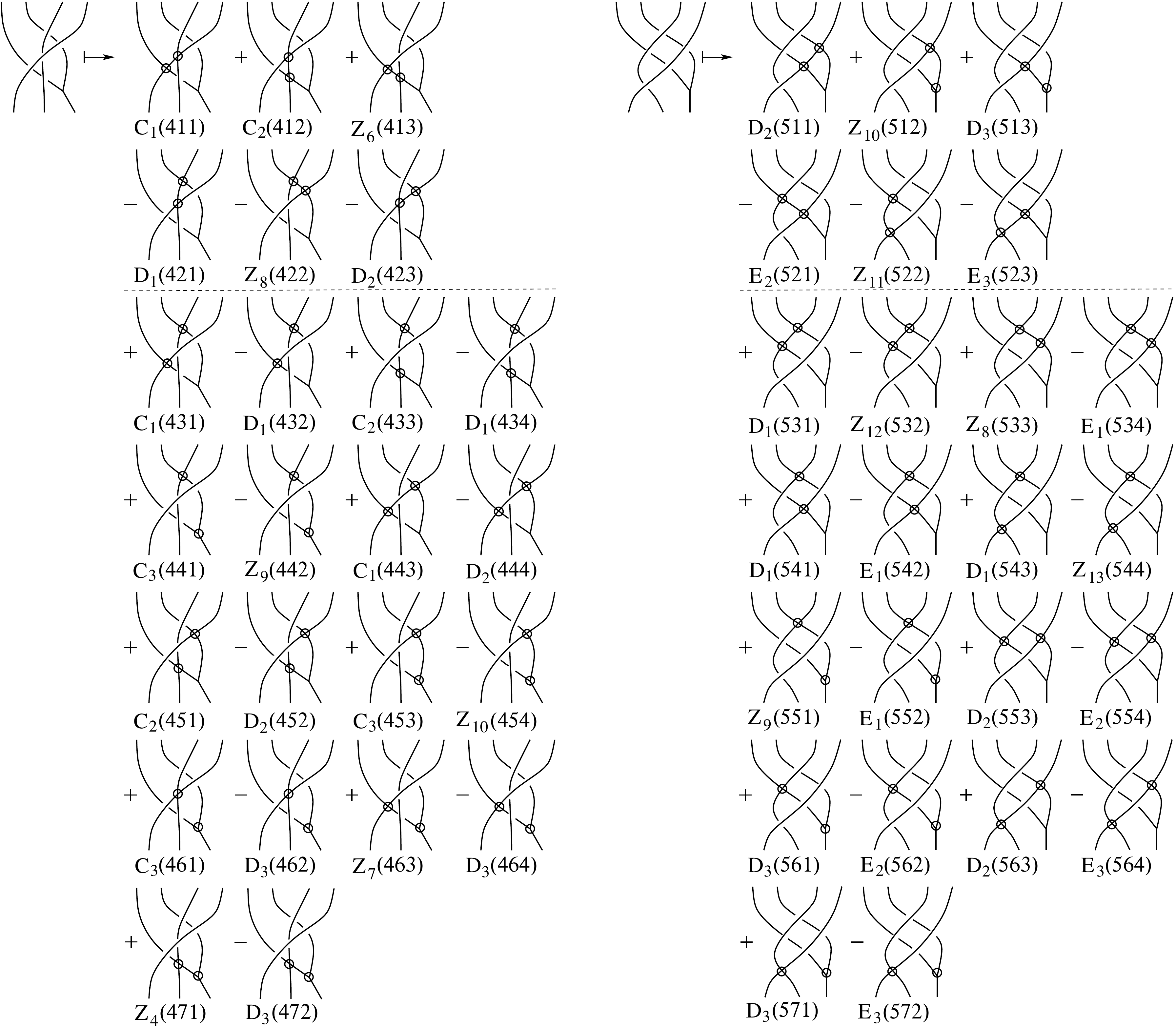}
\end{center}
\caption{}
\label{IIY45}
\end{figure}

\begin{figure}[htb]
\begin{center}
\includegraphics[width=6in]{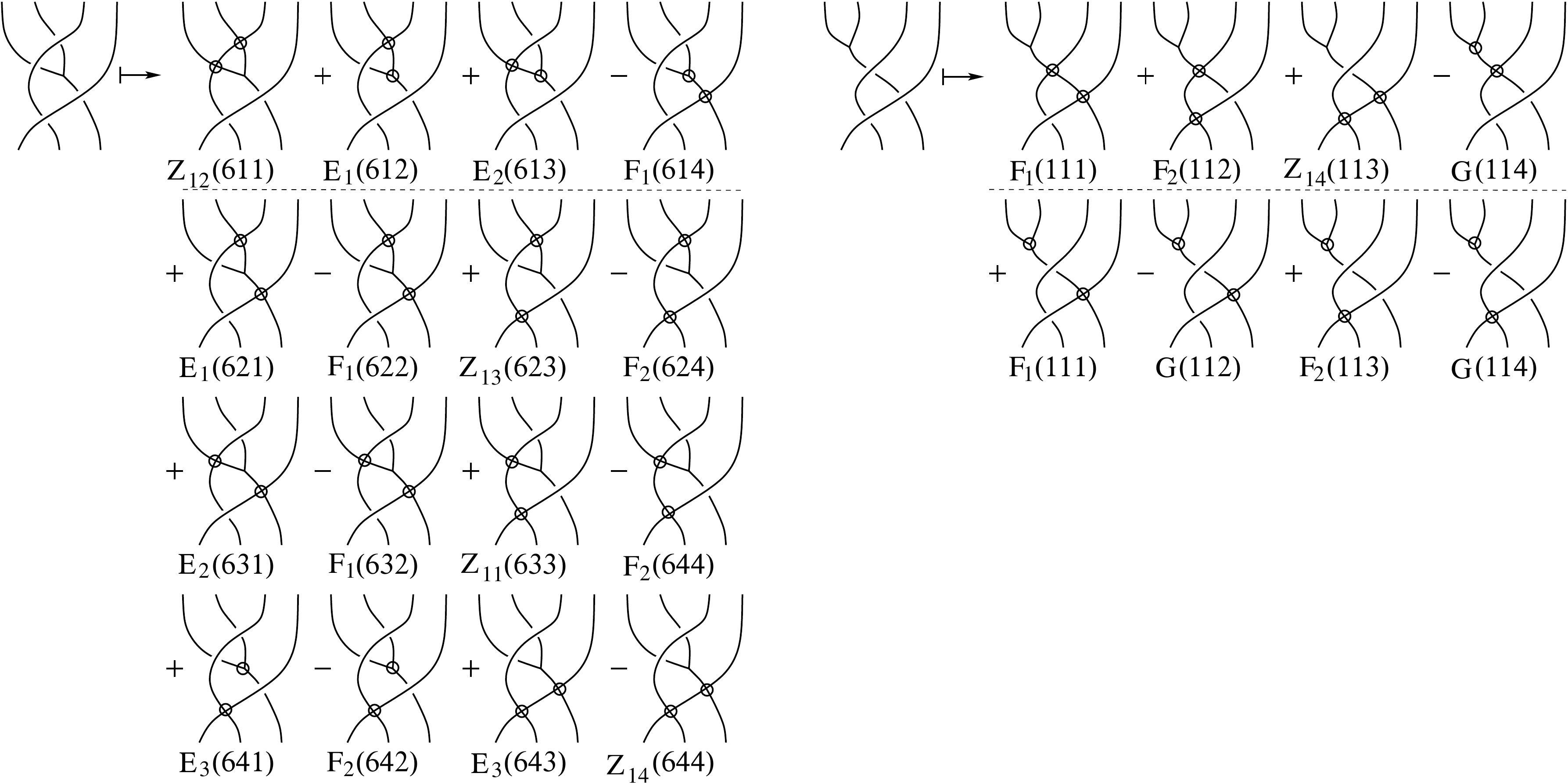}
\end{center}
\caption{}
\label{IIY67}
\end{figure}

In Figure~\ref{IIY3cocylist}, are listed the 3-cocycle conditions that are used to show that the sum of all terms vanish. In Figure~\ref{IIY3ex}, two examples (top two rows) of these 3-cocycle conditions are depicted in the right hand side. Each term is labeled as follows. The first letter $(A_1)$ represents which
3-cocycle condition is used. The leftmost digit of the three digits represents the figure number counted from
Figure~\ref{IIY123} through Figure~\ref{IIY67}. The remaining two digits represent the row-column position. 
Therefore the labels in the terms in Figure~\ref{IIY123} through Figure~\ref{IIY67}
shows that all these terms cancel by the 3-cocycle conditions, and cancelations with opposite signs such as 
the pair depicted in Figure~\ref{IIY3ex} bottom row. 
These canceling terms are labeled by $Z$ with subscripts.
Thus the vanishing of the total sum is checked by labels of diagrams, and the other cases are similarly checked.  

\begin{figure}[htb]
\begin{center}
\includegraphics[width=3.5in]{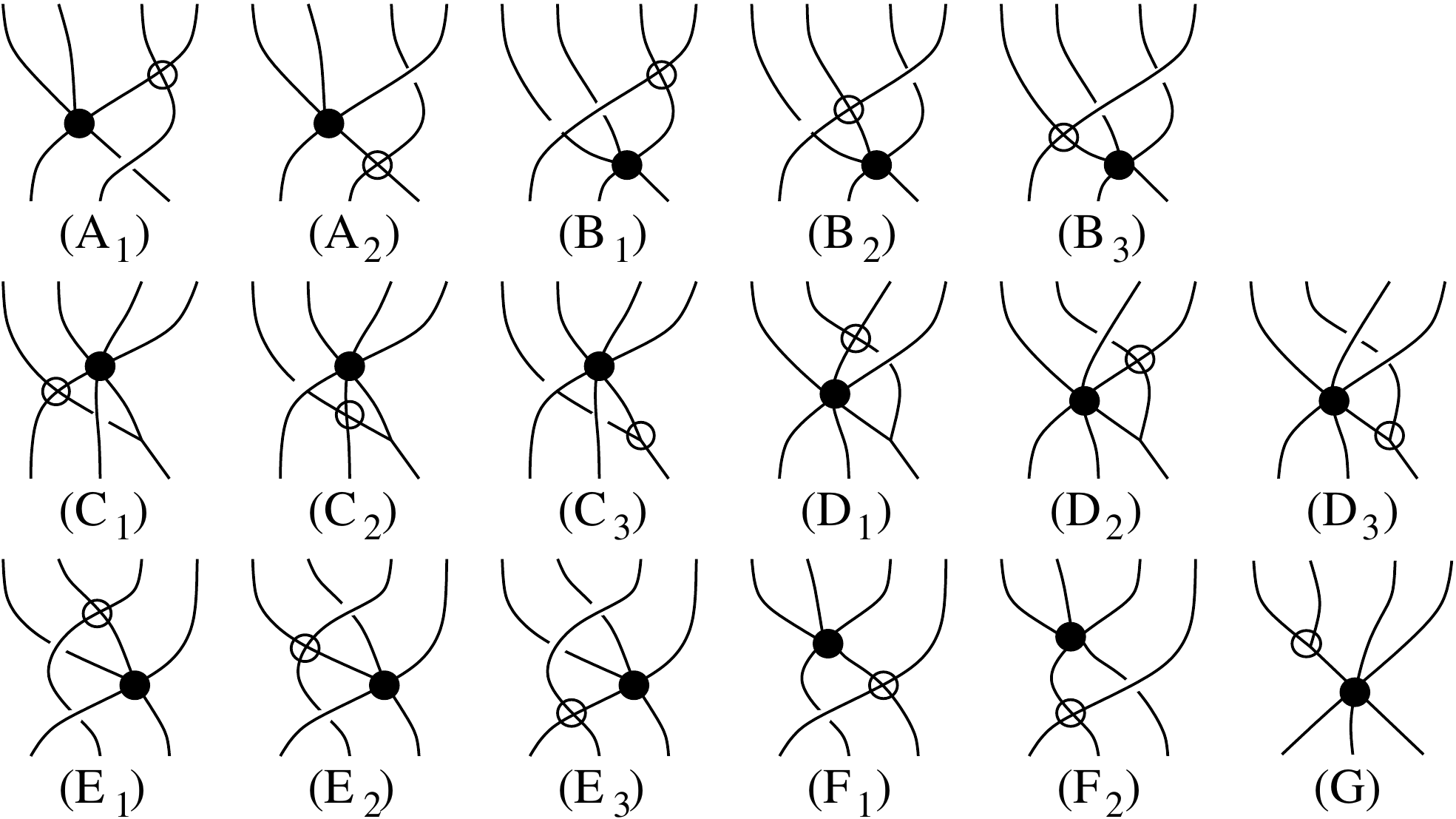}
\end{center}
\caption{}
\label{IIY3cocylist}
\end{figure}

\begin{figure}[htb]
\begin{center}
\includegraphics[width=4in]{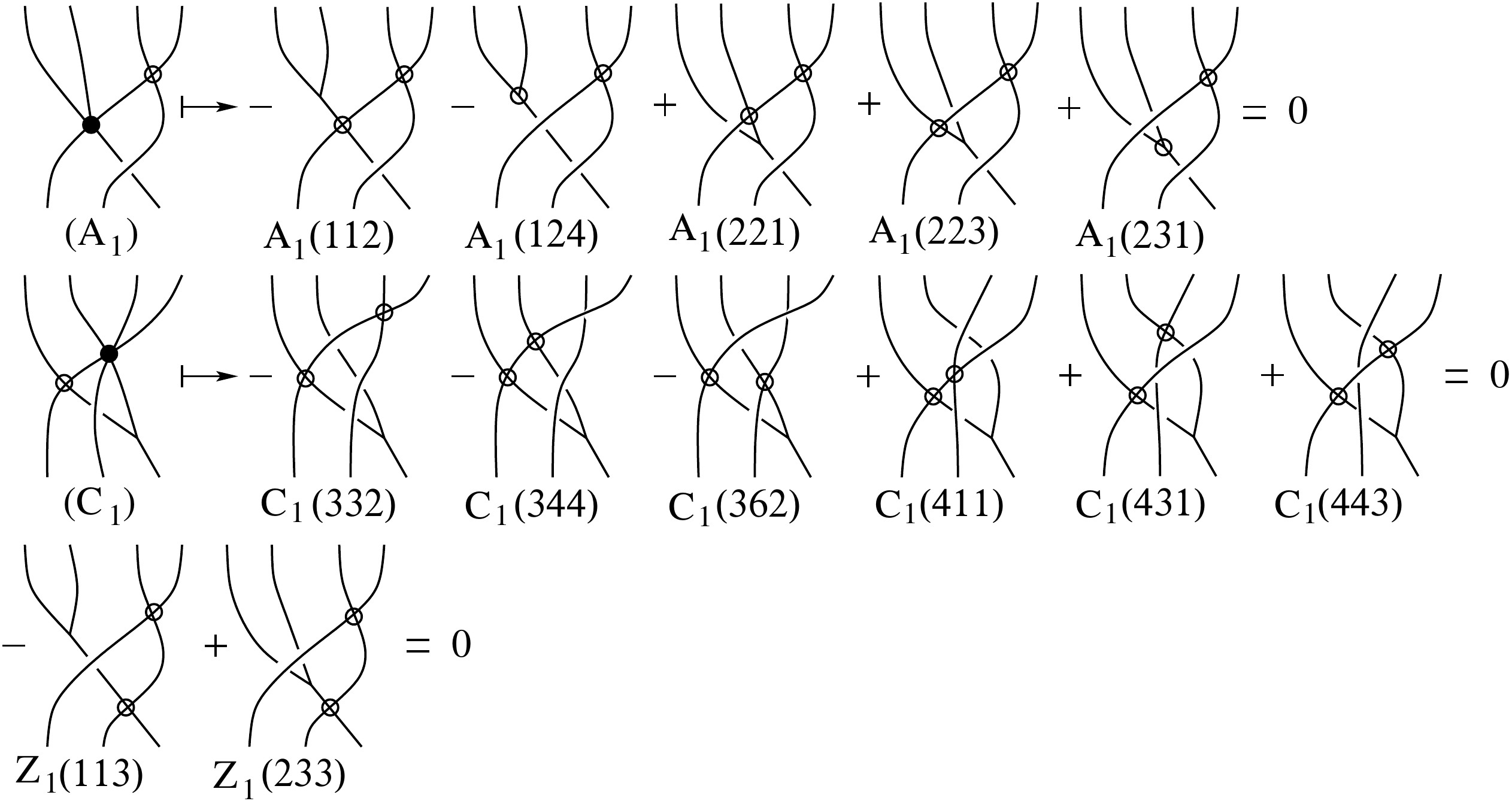}
\end{center}
\caption{}
\label{IIY3ex}
\end{figure}

\end{document}